\documentclass[a4paper]{amsart}
\usepackage{fullpage}
\usepackage{amsmath,amssymb,mathtools,booktabs,tikz}
\usetikzlibrary{arrows.meta,calc,decorations.markings,positioning}
\usepackage[T1]{fontenc}
\usepackage{lmodern}
\usepackage{microtype}
\usepackage{needspace}
\usepackage{tabularx,array}
\usepackage[hypertexnames=false]{hyperref}
\usepackage{doi}
\usepackage{float}

\definecolor{figBlue}{RGB}{211,228,243}
\definecolor{figGreen}{RGB}{221,237,218}
\definecolor{figGold}{RGB}{245,233,203}
\definecolor{figRose}{RGB}{244,218,221}
\definecolor{figRed}{RGB}{178,51,48}
\definecolor{figCobalt}{RGB}{50,87,148}
\definecolor{figInk}{RGB}{43,51,60}

\tikzset{
  figedge/.style={draw=figInk,line width=0.52pt,line join=round},
  figframe/.style={draw=figInk,line width=0.85pt,line join=round},
  figslit/.style={draw=figRed,line width=1.65pt,line cap=round},
  figarrow/.style={-{Latex[length=2mm]},line width=0.9pt},
  figvertex/.style={circle,fill=figInk,inner sep=1.2pt},
  figs/.style={font=\scriptsize}
}

\definecolor{stripblue}{RGB}{221,233,247}
\definecolor{stripgreen}{RGB}{221,239,224}
\definecolor{stripred}{RGB}{180,55,52}
\definecolor{stripviolet}{RGB}{88,84,170}
\definecolor{stripink}{RGB}{40,48,58}

\tikzset{
  stripframe/.style={draw=stripink,line width=0.85pt,line join=round},
  stripedge/.style={draw=stripink,line width=0.60pt,line join=round},
  striparrow/.style={-{Latex[length=2mm]},line width=0.90pt},
  striplab/.style={font=\scriptsize},
  stripdot/.style={circle,fill=stripink,inner sep=1.1pt}
}

\usepackage{graphicx}

\definecolor{proofInk}{HTML}{283644}
\definecolor{proofBlue}{HTML}{32649A}
\definecolor{proofTeal}{HTML}{187C78}
\definecolor{proofRed}{HTML}{B54C4C}
\definecolor{proofGold}{HTML}{967025}
\tikzset{
  proofedge/.style={draw=proofInk,line width=.85pt,line join=round},
  proofarr/.style={-{Latex[length=2.2mm]},line width=.85pt},
  proofdot/.style={circle,fill=proofInk,inner sep=1.4pt},
  prooflab/.style={font=\small},
  proofsmalllab/.style={font=\scriptsize}
}
\newcommand{\prooffigurefont}{\fontsize{9.5}{11.5}\selectfont}

\hypersetup{
    colorlinks=true,
    linkcolor=blue,
    citecolor=blue,
    urlcolor=blue
}
\hypersetup{pdftitle={The Bloch constant conjecture is false},pdfauthor={Teng Zhang}}

\numberwithin{equation}{section}
\newtheorem{thm}{Theorem}[section]
\newtheorem{lem}[thm]{Lemma}
\newtheorem{prop}[thm]{Proposition}
\newtheorem{cor}[thm]{Corollary}

\theoremstyle{plain}
\newtheorem*{blochconjecture}{Bloch constant conjecture}

\theoremstyle{definition}

\newcommand{\D}{\mathbb D}
\newcommand{\C}{\mathbb C}
\newcommand{\R}{\mathbb R}
\newcommand{\Z}{\mathbb Z}
\newcommand{\B}{\mathcal B}
\newcommand{\AG}{A_{\mathrm{AG}}}
\newcommand{\centers}{\mathcal P}
\newcommand{\crit}{\operatorname{Crit}}
\newcommand{\id}{\mathrm{id}}

\newcommand{\dist}{\operatorname{dist}}
\newcommand{\sgn}{\operatorname{sgn}}
\newcommand{\ind}{\mathbf1}
\newcommand{\dbar}{\partial_{\bar z}}

\newcommand{\Aut}{\operatorname{Aut}}

\newcommand{\Rea}{\operatorname{Re}}
\newcommand{\Ima}{\operatorname{Im}}

\title[The Bloch constant conjecture is false]
{The Bloch constant is strictly smaller
than the Ahlfors--Grunsky value}

\author[T.~Zhang]{Teng Zhang}
\address{School of Mathematics and Statistics, Xi'an Jiaotong University,
	Xi'an 710049, P.~R.~China}
\email{teng.zhang@stu.xjtu.edu.cn}

\begin{document}
\begin{abstract}
We prove that the classical Bloch constant is strictly smaller
than the Ahlfors--Grunsky value, disproving the famous Bloch constant conjecture. The Lean 4 formalization of the main results are also provided.
\end{abstract}
	\maketitle

	\section{Introduction}
	\label{sec:intro}

	Let $\mathcal O(\D)$ be the set of holomorphic functions on the unit disk
	$\D=\{z\in\C:|z|<1\}$.
	For a nonconstant function $f\in\mathcal O(\D)$, let $b(f)$
	denote the supremum of the radii of Euclidean disks onto which
	$f$ maps some subdomain of $\D$ conformally. Equivalently,
	$b(f)$ is the supremum of all $r>0$ for which there exist
	$w\in\C$ and a holomorphic map
	$\psi:D(w,r)\to\D$ satisfying
	$
	f\circ\psi=\id_{D(w,r)},
	$
	where $D(w,r)=\{\zeta\in\C:|\zeta-w|<r\}$.
	The classical Bloch constant is defined by
	\begin{equation*}
		\B=\inf\bigl\{b(f):f\in\mathcal O(\D),\
		f(0)=0,\ f'(0)=1\bigr\}.
	\end{equation*}
	Bloch's theorem~\cite{Bloch} guarantees that $\B>0$.
	Landau~\cite{Landau} considered the related problem of finding
	disks contained in the image of a holomorphic function,
	which leads to the Landau constant problem.

	Ahlfors and Grunsky~\cite{AG} constructed a function whose inverse
	Riemann surface is tiled by equilateral triangles, with simple
	branch points lying over a triangular lattice. Their construction
	yields the upper bound
	\begin{equation}\label{eq:AG}
		\B\leq\AG:=
		\frac{\Gamma(1/3)\Gamma(11/12)}
		{\sqrt{1+\sqrt3}\,\Gamma(1/4)}
		=0.471861653452681\ldots,
	\end{equation}
	where $\Gamma$ denotes the Euler gamma function. They conjectured that equality holds in \eqref{eq:AG}, which is known as the \emph{Bloch constant conjecture}.

	\begin{blochconjecture}
		The Bloch constant satisfies $
		\B=\AG.$
	\end{blochconjecture}

	Progress on lower bounds for Bloch constant is summarized
	in Table~\ref{tab:bloch-lower}.
	\begin{table}[H]
		\centering
		\renewcommand{\arraystretch}{1.15}
		\begin{tabular}{|c|c|}
			\hline
			Lower bound for $\B$ & Reference \\
			\hline
			$\B\geq\sqrt3/4$
			& Ahlfors~\cite{Ahlfors} \\
			$\B>\sqrt3/4$
			& Heins~\cite{Heins} \\
			$\B\geq\sqrt3/4+10^{-14}$
			& Bonk~\cite{Bonk} \\
			$\B\geq\sqrt3/4+2\cdot10^{-4}$
			& Chen--Gauthier~\cite{CG} \\
			$\B\geq\sqrt3/4+3\cdot10^{-4}$
			& Xiong~\cite{Xiong} \\
			$\B\geq\sqrt3/4+0.0153$
			& Wikstr\"om~\cite{Wikstrom}  \\
			\hline
		\end{tabular}
		\caption{A summary of lower bounds for Bloch constant $\B$.}
		\label{tab:bloch-lower}
	\end{table}
For the spherical metric, Bonk and Eremenko~\cite{BE}
proved that every nonconstant meromorphic function on $\C$
univalently covers spherical disks of radii arbitrarily
close to $\arctan\sqrt{8}$, and that this constant is sharp.
They further showed that the sharp constant is $\pi/2$
when all critical points are multiple. Unlike the classical Bloch constant problem, their
covering theorem requires no derivative normalization.

We recall the classical Ahlfors--Grunsky function
\cite{AG}, defined by an anaytic continuation by Schwarz reflectionr
of the conformal map $g$ 
from an
equilateral hyperbolic triangle with all three angles
equal to $\pi/6$ onto a Euclidean equilateral triangle,
and normalized by $g(0)=0$ and $g'(0)=1$.
For this construction, see Yamada \cite[p.~401]{Yamada}.

Minda~\cite{Minda} studied more general Bloch constants, while
Yanagihara~\cite{YanagiharaExtremal} gave necessary
conditions for extremal functions. Goodman~\cite{Goodman} considered variations of
individual branch points. Yamada~\cite{Yamada} showed that
any sufficiently small nonzero displacement of a simple
branch point of the Ahlfors--Grunsky inverse Riemann surface
strictly increases the normalized Bloch radius, provided
the other branch points and the boundary remain fixed.
Yanagihara~\cite{Yanagihara} studied second-order variations
induced by affine transformations of the full image surface. Further local
minimality results were obtained by A. Baernstein II and J. Vinson~\cite{BV}.

In this paper, our main result disproves the Bloch constant conjecture.
\begin{thm}\label{thm:main}
There exist an Ahlfors--Grunsky function $g$, constants $\delta,c>0$, and
normalized holomorphic functions $f_s$ on $\D$, $0<s<\delta$, such that
\begin{equation}\label{eq:main}
 b(g)=\AG,\qquad b(f_s)\leq\AG(1-cs),\qquad
 f_s\longrightarrow g\quad\text{locally uniformly as }s\downarrow0.
\end{equation}
\end{thm}

Since $\B\leq b(f_s)$ for every $s\in(0,\delta)$,
Theorem~\ref{thm:main} gives the following corollary.

\begin{cor}\label{cor}
The Bloch constant satisfies $\B<\AG$.
\end{cor}

We also obtain an explicit upper bound for  $\B$.

\begin{thm}\label{thm:main2}
The Bloch constant satisfies $\B<\AG-1.3\times10^{-9}$.
\end{thm}

    The upper bound of Theorem \ref{thm:main2} can certainly be improved using the methods developed in this paper, but this is not the essential point.

The Lean 4 formalizations of Theorem \ref{thm:main}, Corollary \ref{cor}, and Theorem \ref{thm:main2} can be found in \cite{Zha26}.

\medskip
\noindent\textbf{Sketch of the proof.}
We begin with a translation surface of genus two tiled by equilateral
triangles. Its developing map on the universal cover gives an
Ahlfors--Grunsky function after normalization at a face center.
A horizontal cylinder on the surface lifts to a strip in $\D$.
We shear one such lift and extend the deformation to
the two complementary components. Solving the resulting Beltrami
equation gives holomorphic maps $\widetilde f_t:\D\to\C$.

An inverse-branch argument bounds the radius of every univalent disk
by the largest circumradius of the deformed faces. This gives
\[
b(\widetilde f_t)\leq\frac{1}{\sqrt3}+O(t^2).
\]
The conformal normalization has a nonzero first variation at some
face center outside the sheared strip. To find such a center, we first
show that the variation cannot be infinitesimally trivial on
both complementary components. A boundary expansion, together with
the fact that every point of $\D$ lies within a bounded hyperbolic
distance of a face center, then gives the required center.
Normalizing there and choosing the sign of $t$ produces a decrease
of first order in the derivative, while the change in the
circumradius is only of second order. This yields
$b(f_s)\leq\AG(1-cs)$ for the normalized functions $f_s$,
and they converge locally
uniformly to an Ahlfors--Grunsky function as $s\downarrow0$.

For the explicit bound, we use a deformation of the compact surface
that shears every lift of one cylinder. A Green-function comparison
gives an estimate with an explicit quadratic remainder, and its
first-order coefficient is expressed in terms of finitely many
integrals on the compact surface. Combining the numerical 
estimates in Lemma~\ref{lem:intervalresult} with the remainder estimate gives
Proposition~\ref{prop:quantitative}. The qualitative proof the main result
does not
use this numerical computation.

\medskip
\noindent\textbf{Organization of the paper.}
In Section~\ref{sec:prelim} we collect the analytic facts used in the proof.
In Sections~\ref{sec:surface} and~\ref{sec:strip} we construct the translation
surface and the deformation of a single lifted cylinder.
In Section~\ref{sec:variation} we study the first variation of the conformal
normalization and find a face center with nonzero variation.
In Section~\ref{sec:radius} we bound the radii of all univalent disks,
and Section~\ref{sec:descent} we complete the proof of
Theorem~\ref{thm:main}. In Section~\ref{sec:quantitative} the
compact surface deformation is treated and the explicit bound is derived.

\medskip
\noindent\textbf{Acknowledgments and AI tools disclosure.}
The author thanks Professors A.~Eremenko and Z.~Xu for their
encouragement and support.

Teng Zhang is supported by the China Scholarship Council, the Young
Elite Scientists Sponsorship Program for PhD Students of the China
Association for Science and Technology, and the Fundamental Research
Funds for the Central Universities at Xi'an Jiaotong University
(Grant No.~xzy022024045).

ChatGPT was used for English-language editing and proofreading, and
as an exploratory tool for discussing possible approaches to selected
parts of the argument. The mathematical arguments and proofs in the
original manuscript were independently developed, checked, and
written by the author.

The Lean~4 formalization in~\cite{Zha26} was generated using
OpenAI's Codex and checked by the Lean kernel.
\section{Classical analytic facts and notation}
\label{sec:prelim}

We equip $\D$ with the Poincar\'e metric $2|dz|/(1-|z|^2)$,
normalized to have Gaussian curvature $-1$. We denote its induced
distance by $d_{\D}$; explicitly,
\[
d_{\D}(z,w)=2\operatorname{arctanh}
\left|\frac{z-w}{1-\overline w z}\right|,
\qquad z,w\in\D.
\]
For an open set $\Omega\subset\C$
and $1\leq p\leq\infty$, let $L^p(\Omega)$ and
$W^{1,p}(\Omega)$ denote the usual Lebesgue and Sobolev
spaces of complex-valued functions, respectively.
We write $W^{1,p}_{\mathrm{loc}}(\Omega)$ for the
corresponding local Sobolev space, and
$W^{1,2}_0(\Omega)$ for the closure of
$C_c^\infty(\Omega)$ in $W^{1,2}(\Omega)$.

In the complex coordinate $z=x+iy$, let
$dA=dx\,dy$ denote Euclidean area measure, and
define the Wirtinger derivatives by
\[
\partial_z=\frac12(\partial_x-i\partial_y),
\qquad
\partial_{\bar z}=\frac12(\partial_x+i\partial_y).
\]
Derivatives of Sobolev functions are understood
in the distributional sense.

An orientation-preserving homeomorphism
$h:\Omega\to\Omega'$ between plane domains is
called \emph{quasiconformal} if
$h\in W^{1,2}_{\mathrm{loc}}(\Omega)$ and, for
some $0\leq k<1$,
\[
|h_{\bar z}|\leq k|h_z|
\qquad\text{almost everywhere on }\Omega.
\]
A \emph{Beltrami coefficient} on $\Omega$ is a
function $\mu\in L^\infty(\Omega)$ with
$\|\mu\|_\infty<1$. The equation
$
h_{\bar z}=\mu h_z
$
is called the \emph{Beltrami equation} and is
understood almost everywhere.

For $G\in W^{1,2}_{\mathrm{loc}}(\Omega)$,
we denote its Jacobian by
\[
J_G=|G_z|^2-|G_{\bar z}|^2.
\]
For a holomorphic function $F$, we write
$\crit(F)=\{z:F'(z)=0\}$ for its critical set.

For $0<\delta<1$, let
\[
L^\infty_\delta(\C)
=
\{\mu\in L^\infty(\C):\|\mu\|_\infty<\delta\}.
\]
For $1<p<\infty$, let $N_p$ denote the operator
norm on $L^p(\C)$ of the Beurling transform
\[
Tg(z)=-\frac1\pi\operatorname{pv}
\int_\C\frac{g(\zeta)}{(z-\zeta)^2}\,dA(\zeta).
\]
Recall that $N_2=1$ and $N_p\to1$ as $p\downarrow2$.

For $R>0$ and $p>2$, set $\alpha_p=1-2/p$
and write $D_R=D(0,R)$. Let $B_{R,p}$ be the
Banach space of functions
$v\in C^{\alpha_p}(\overline{D_R})$ satisfying
$v(0)=0$ and $v_{\bar z}\in L^p(D_R)$,
equipped with the norm
\[
\|v\|_{B_{R,p}}
=
\sup_{\substack{z,w\in\overline{D_R}\\z\ne w}}
\frac{|v(z)-v(w)|}{|z-w|^{\alpha_p}}
+
\|v_{\bar z}\|_{L^p(D_R)}.
\]

We will use the following results of Ahlfors and Bers \cite{AB}.

\begin{lem}[{\cite[Theorems 6 and 11]{AB}}]\label{lem:AB}
\begin{enumerate}
\item
For every $\mu\in L^\infty_1(\C)$,
there exists a unique quasiconformal
homeomorphism
$h^\mu:\widehat{\C}\to\widehat{\C}$
fixing $0$, $1$, and $\infty$ and satisfying
$
(h^\mu)_{\bar z}=\mu(h^\mu)_z
$
almost everywhere on $\C$.
Likewise, every Beltrami coefficient on $\D$
determines a unique quasiconformal
self-homeomorphism of $\D$ fixing $0$ and $1$.

\item 
Let $0<\delta<1$ and $p>2$ satisfy
$\delta N_p<1$. Suppose that
$\Lambda\subset\C$ is a domain and that
$\lambda\mapsto\mu_\lambda$ is holomorphic
from $\Lambda$ into $L^\infty_\delta(\C)$.
Then, for every $R>0$,
$
\lambda\mapsto h^{\mu_\lambda}|_{\overline{D_R}}
$
is holomorphic from $\Lambda$ into $B_{R,p}$.
\end{enumerate}
\end{lem}

\begin{lem}[{\cite[Lemma 1]{AB}}]\label{lem:Weyl}
A continuous function in $W^{1,1}_{\mathrm{loc}}(\Omega)$ satisfying
$\dbar v=0$ in distributions is holomorphic on $\Omega$.
\end{lem}

\begin{prop}\label{prop:AB}
Let $\Lambda\subset\C$ be a domain containing
$0$, and suppose that
$\lambda\mapsto\mu_\lambda$ is holomorphic
from $\Lambda$ into $L^\infty_1(\C)$,
with $\mu_0=0$.
Fix three distinct points
$\zeta_1,\zeta_2,\zeta_3\in\partial\D$,
and let $H_\lambda$ be the unique
quasiconformal homeomorphism of
$\widehat{\C}$ with Beltrami coefficient
$\mu_\lambda$ fixing these three points.

\begin{enumerate}
\item
For every $R>1$, there exist a neighborhood
$\Lambda_R$ of $0$ and an exponent $p>2$
such that
\[
\lambda\longmapsto H_\lambda|_{D_R}
\]
is holomorphic from $\Lambda_R$ into
$W^{1,p}(D_R)$.

\item
Suppose that $\mu_\lambda=0$ almost everywhere
on a fixed open set
$\mathcal V\subset D_R$ for every
$\lambda\in\Lambda_R$.
Then $(\lambda,z)\mapsto H_\lambda(z)$
is jointly holomorphic on
$\Lambda_R\times\mathcal V$.
In particular,
\[
\partial_\lambda^j\partial_z^m
H_\lambda(z)
=
\partial_z^m\partial_\lambda^j
H_\lambda(z),
\qquad j,m\geq0,
\]
and all these derivatives depend locally
uniformly on $(\lambda,z)$.

\item
Suppose, in addition, that for real
$t\in\Lambda_R$ the coefficients satisfy
\[
\mu_t(z)=
\frac{z^2}{\bar z^2}
\overline{\mu_t(1/\bar z)}
\]
almost everywhere on $\C\setminus\{0\}$.
Then $H_t$ restricts to a quasiconformal
self-homeomorphism of $\D$.

Moreover, the variation
$V=\partial_tH_t|_{t=0}$ satisfies
\[
H_t=\id+tV+O(t^2)
\quad\text{in }W^{1,p}(D_R),
\]
\[
\dbar V=
\left.\partial_t\mu_t\right|_{t=0}
\]
in the distributional sense, and
\[
\Rea(\bar\zeta V(\zeta))=0,
\qquad \zeta\in\partial\D.
\]
\end{enumerate}
\end{prop}

\begin{proof}
\emph{For (1).}
Fix $R>1$. Since $\mu_0=0$, we may choose
$0<\delta<1$ and a neighborhood
$\Lambda_0$ of $0$ such that
$\mu_\lambda\in L^\infty_\delta(\C)$
for every $\lambda\in\Lambda_0$.
Choose $p>2$ sufficiently close to $2$
that $\delta N_p<1$.

Write $h_\lambda=h^{\mu_\lambda}$.
By Lemma~\ref{lem:AB}(2),
$\lambda\mapsto h_\lambda$ is holomorphic
with values in $B_{R',p}$ for every $R'>0$.

We first show that, whenever $R<R'$,
the restriction map
$
B_{R',p}\longrightarrow W^{1,p}(D_R)
$
is continuous. Choose
$\chi\in C_c^\infty(D_{R'})$ with
$\chi=1$ on a neighborhood of
$\overline{D_R}$.
For $v\in B_{R',p}$, the definition
of the norm gives
\[
\|v\|_{L^\infty(D_{R'})}
+
\|v_{\bar z}\|_{L^p(D_{R'})}
\leq C\|v\|_{B_{R',p}}.
\]
Applying the $L^p$ boundedness of
the Beurling transform to $\chi v$ yields
\[
\begin{aligned}
\|(\chi v)_z\|_{L^p(\C)}
&\leq N_p\|(\chi v)_{\bar z}\|_{L^p(\C)}\\
&\leq C\left(
\|v_{\bar z}\|_{L^p(D_{R'})}
+\|v\|_{L^\infty(D_{R'})}
\right).
\end{aligned}
\]
Consequently,
\[
\|v\|_{W^{1,p}(D_R)}
\leq C\|v\|_{B_{R',p}}.
\]
It follows that $\lambda\mapsto
h_\lambda|_{D_R}$ is holomorphic
with values in $W^{1,p}(D_R)$.

Since $p>2$, the Sobolev embedding theorem shows
that $h_\lambda(\zeta_j)$ depends
holomorphically on $\lambda$ for
$j=1,2,3$.
Let $M_\lambda$ be the unique M\"obius
transformation satisfying
\[
M_\lambda(h_\lambda(\zeta_j))
=\zeta_j,\qquad j=1,2,3.
\]
Its coefficients depend holomorphically
on $\lambda$, and $M_0=\id$.
By uniqueness,
\[
H_\lambda=M_\lambda\circ h_\lambda.
\]
After shrinking $\Lambda_0$ if necessary,
$M_\lambda$ has no pole on a neighborhood
of $h_\lambda(\overline{D_R})$.
As $W^{1,p}(D_R)$ is a Banach algebra
for $p>2$, composition with this
holomorphic family of M\"obius
transformations preserves holomorphic
dependence in $W^{1,p}(D_R)$.
This proves (1).

\emph{For (2).}
Under the hypothesis of (2), the
Beltrami equation gives
$(H_\lambda)_{\bar z}=0$ almost everywhere
on $\mathcal V$. By Lemma~\ref{lem:Weyl},
$H_\lambda$ is holomorphic on $\mathcal V$.
Moreover, (1) and the Sobolev embedding theorem imply
that $\lambda\mapsto H_\lambda(z)$ is
holomorphic for each fixed
$z\in\mathcal V$.
Joint holomorphy follows from Hartogs'
theorem, and the assertions concerning
derivatives follow from the Cauchy
integral formula.

\emph{For (3).}
Let $J(z)=1/\bar z$.
The reflection hypothesis implies that
$J\circ H_t\circ J$ has the same
Beltrami coefficient as $H_t$.
Both maps fix $\zeta_1,\zeta_2,\zeta_3$,
so uniqueness gives
\[
J\circ H_t\circ J=H_t.
\]
Thus $H_t$ preserves $\partial\D$.
Since $H_0=\id$ and $H_t$ depends
continuously on $t$, we have
$H_t(\D)=\D$ for all sufficiently
small real $t$.

By (1), the family admits a Taylor
expansion in $W^{1,p}(D_R)$,
which gives the asserted expansion
for $H_t$. Differentiating
$
(H_t)_{\bar z}=\mu_t(H_t)_z
$
at $t=0$, using $H_0=\id$ and
$\mu_0=0$, yields
\[
\dbar V=
\left.\partial_t\mu_t\right|_{t=0}.
\]
Finally, $|H_t(\zeta)|=1$ for
$\zeta\in\partial\D$.
Differentiating at $t=0$ gives
\[
2\Rea(\bar\zeta V(\zeta))=0,
\]
as required.
\end{proof}

\begin{lem}[{\cite[Theorem 5.5.1]{AIM}}]\label{lem:Stoilow}
Let $\Omega$ be a plane domain, let $h:\Omega\to\Omega'$ be a
quasiconformal homeomorphism, and let $G\in W^{1,2}_{\mathrm{loc}}(\Omega)$
be continuous. If $G_{\bar z}=\mu G_z$ and $h_{\bar z}=\mu h_z$ almost
everywhere for the same Beltrami coefficient $\mu$, then $G\circ h^{-1}$
is holomorphic on $\Omega'$.
\end{lem}

We state the normalization of the Ahlfors--Grunsky
triangle map that will be used below.

\begin{lem}\label{lem:AGtriangle}
Let $T\subset\D$ be a regular hyperbolic triangle
centered at $0$ with angles $\pi/6$, and let
$\Delta$ be a Euclidean equilateral triangle of
side length $1$. If $\Phi:T\to\Delta$ is a
conformal map taking the vertices of $T$ to those
of $\Delta$ in cyclic order, then
\begin{equation}\label{eq:centerconstant}
|\Phi'(0)|=c_0=\frac{1}{\sqrt3\,\AG}.
\end{equation}
\end{lem}

\begin{proof}
Put $\omega=e^{2\pi i/3}$ and
$\sigma=(1+\sqrt3)^{-1/2}$.
In the classical Ahlfors--Grunsky construction
\cite{AG}, as described in
Yamada~\cite[p.~401]{Yamada}, the hyperbolic
triangle $S$ with vertices
$\sigma,\omega\sigma,\omega^2\sigma$
is mapped conformally onto the Euclidean
triangle with vertices
$\AG,\omega\AG,\omega^2\AG$.
The map $g$ is normalized by
$g(0)=0$ and $g'(0)=1$.

Since $T$ and $S$ are regular hyperbolic
triangles with the same angles and center,
they differ by a rotation of $\D$.
The target triangle of $g$ has circumradius
$\AG$ and hence side length $\sqrt3\,\AG$.
An orientation-preserving Euclidean similarity
taking this triangle onto $\Delta$ therefore
has derivative of modulus $1/(\sqrt3\,\AG)$.

After choosing the rotation and similarity
to match the prescribed vertices, their
composition with $g$ agrees with $\Phi$.
Indeed, conformal maps between two Jordan
triangles are uniquely determined by the
images of their three vertices.
Since the source rotation has derivative
of modulus $1$, the chain rule gives
\eqref{eq:centerconstant}.
\end{proof}

\section{An equilateral translation surface}
\label{sec:surface}

We construct the developing map from a branched double cover
of a flat torus. Set
\begin{equation*}
h=\frac{\sqrt3}{2},\qquad
\tau=\frac12+ih,\qquad
L=\Z+\tau\Z,\qquad
L_0=2\Z+\tau\Z.
\end{equation*}
Let $E=\C/L_0$. Since $L_0\cap\R=2\Z$, the segment
$[0,1]$ projects to an embedded arc in $E$.
Take two copies of $E$, cut each along this arc, and
identify the opposite banks of the two slits crosswise
by translations. Denote the resulting compact Riemann
surface by $X$ and the natural projection by $P:X\to E$.
The holomorphic differential $dw$ on $E$ pulls back to
a holomorphic differential $\omega=P^*dw$ on $X$.

\begin{lem}\label{lem:surface}
The surface $X$ has genus two, and $\omega$ has exactly
two zeros, both simple. Moreover, $X$ admits a
triangular cell decomposition with eight equilateral faces of side
length $1$, whose vertices are the zeros of $\omega$.

Let $\pi:\D\to X$ be a universal covering map, and let
$F$ be a primitive of $\pi^*\omega$. Then $F$ is
holomorphic on $\D$, and its critical points are
precisely the lifts of the two zeros of $\omega$,
each of local degree two. On every closed lifted face,
$F$ is a homeomorphism onto a Euclidean equilateral
triangle of side length $1$.
\end{lem}

\begin{proof}
The projection $P$ is a connected double cover, unbranched away from
the endpoints of the slit. A small loop around either endpoint interchanges
the two sheets. Thus $P$ has precisely two simple branch points, lying
over $0$ and $1$. In a flat coordinate $w$ on $E$, with $w_v\in\{0,1\}$
the relevant branch value, one may choose a local coordinate $\xi$ on $X$
so that
\[
P(\xi)=w_v+\xi^2,\qquad \omega=2\xi\,d\xi.
\]
Hence $\omega$ has a simple zero at each branch point and has no other
zeros. Since $E$ has genus one, Riemann--Hurwitz gives
\[
2g(X)-2=2\bigl(2g(E)-2\bigr)+2=2,
\]
so $g(X)=2$.

We now describe the triangular faces. A fundamental parallelogram for
$L_0$ has vertices $0,2,2+\tau,\tau$. Its four unit equilateral triangles
have counterclockwise vertex triples
\[
\begin{aligned}
A_j&=(j,j+1,j+\tau),\\
B_j&=(j+1+\tau,j+\tau,j+1),\qquad j=0,1.
\end{aligned}
\]
Indeed, $|\tau|=|1-\tau|=1$, so every side has length $1$.
Figure~\ref{fig:fundamental} shows the four faces.
\begin{figure}[H]
\centering
\begin{tikzpicture}[x=2.05cm,y=2.05cm,figs]
  \def\hh{0.866025403784}
  \coordinate (O) at (0,0);
  \coordinate (M) at (1,0);
  \coordinate (R) at (2,0);
  \coordinate (T) at (.5,\hh);
  \coordinate (N) at (1.5,\hh);
  \coordinate (S) at (2.5,\hh);

  \fill[figBlue]  (O)--(M)--(T)--cycle;
  \fill[figGreen] (N)--(T)--(M)--cycle;
  \fill[figRose]  (M)--(R)--(N)--cycle;
  \fill[figGold]  (S)--(N)--(R)--cycle;

  \draw[figframe] (O)--(R)--(S)--(T)--cycle;
  \draw[figedge] (M)--(T) (M)--(N) (R)--(N);
  \draw[figslit] (O)--(M);

  \node at (.50,.275) {$A_0$};
  \node at (1.00,.57) {$B_0$};
  \node at (1.50,.275) {$A_1$};
  \node at (2.00,.57) {$B_1$};

  \node[figvertex,label={[font=\scriptsize]below left:$0$}]
    at (O) {};
  \node[figvertex,label={[font=\scriptsize]below:$1$}]
    at (M) {};
  \node[figvertex,label={[font=\scriptsize]below:$2$}]
    at (R) {};
  \node[figvertex,label={[font=\scriptsize]above left:$\tau$}]
    at (T) {};
  \node[figvertex,label={[font=\scriptsize]above:$1+\tau$}]
    at (N) {};
  \node[figvertex,label={[font=\scriptsize]above right:$2+\tau$}]
    at (S) {};

  \draw[<->,thin] (-.24,0)--(-.24,\hh);
  \node[left] at (-.24,.433) {$h$};
\end{tikzpicture}
\caption{The four faces of $E$.}
\label{fig:fundamental}
\end{figure}
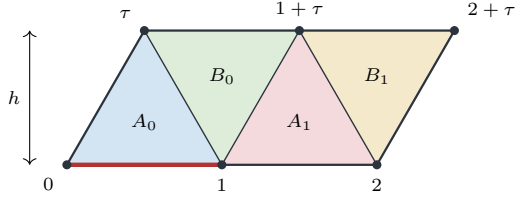

Identifying opposite sides gives the torus $E$, as in
Figure~\ref{fig:torus}. The only vertices in the quotient are
$0$ and $1$ modulo $L_0$.
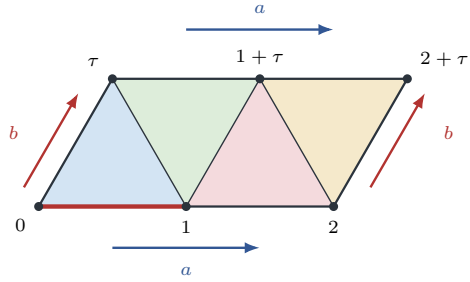
\begin{figure}[H]
\centering
\begin{tikzpicture}[x=1.95cm,y=1.95cm,figs,>=Latex]
  \def\hh{0.866025403784}
  \coordinate (O) at (0,0);
  \coordinate (M) at (1,0);
  \coordinate (R) at (2,0);
  \coordinate (T) at (.5,\hh);
  \coordinate (N) at (1.5,\hh);
  \coordinate (S) at (2.5,\hh);

  \fill[figBlue]  (O)--(M)--(T)--cycle;
  \fill[figGreen] (N)--(T)--(M)--cycle;
  \fill[figRose]  (M)--(R)--(N)--cycle;
  \fill[figGold]  (S)--(N)--(R)--cycle;

  \draw[figframe] (O)--(R)--(S)--(T)--cycle;
  \draw[figedge] (M)--(T) (M)--(N) (R)--(N);
  \draw[figslit] (O)--(M);

  \draw[figarrow,figCobalt] (.5,-.28)--(1.5,-.28);
  \draw[figarrow,figCobalt] (1,1.20)--(2,1.20);
  \draw[figarrow,figRed] (-.10,.13)--(.27,.77);
  \draw[figarrow,figRed] (2.25,.13)--(2.62,.77);

  \node[figCobalt] at (1,-.44) {$a$};
  \node[figCobalt] at (1.5,1.34) {$a$};
  \node[figRed] at (-.17,.5) {$b$};
  \node[figRed] at (2.78,.5) {$b$};

  \node[figvertex,label={[font=\scriptsize]below left:$0$}]
    at (O) {};
  \node[figvertex,label={[font=\scriptsize]below:$1$}]
    at (M) {};
  \node[figvertex,label={[font=\scriptsize]below:$2$}]
    at (R) {};
  \node[figvertex,label={[font=\scriptsize]above left:$\tau$}]
    at (T) {};
  \node[figvertex,label={[font=\scriptsize]above:$1+\tau$}]
    at (N) {};
  \node[figvertex,label={[font=\scriptsize]above right:$2+\tau$}]
    at (S) {};
\end{tikzpicture}
\caption{The torus identifications.}
\label{fig:torus}
\end{figure}

The segment $[0,1]$ is one edge of this cell decomposition. Cutting two
copies of $E$ along this edge and gluing them crosswise leaves all eight
faces intact; see Figure~\ref{fig:slit-gluing}.
\begin{figure}[H]
\centering
\begin{tikzpicture}[x=1.65cm,y=1.65cm,figs,>=Latex]
  \def\hh{0.866025403784}

  \foreach \xshift/\idx in {0/1,4.35/2}{
    \begin{scope}[shift={(\xshift,0)}]
      \fill[figBlue!30]
        (-.65,-.55)--(1.35,-.55)--(1.85,.3160254)
        --(-.15,.3160254)--cycle;
      \draw[figframe]
        (-.65,-.55)--(1.35,-.55)--(1.85,.3160254)
        --(-.15,.3160254)--cycle;

      \draw[white,line width=5pt] (0,0)--(1,0);
      \fill[figInk] (0,0) circle (.023);
      \fill[figInk] (1,0) circle (.023);

      \node[above] at (.6,.41) {$E^{(\idx)}$};
      \node[below] at (0,-.06) {$0$};
      \node[below] at (1,-.06) {$1$};
    \end{scope}
  }

  \draw[figRed,line width=1.4pt] (0,.037)--(1,.037);
  \draw[figCobalt,line width=1.4pt] (0,-.037)--(1,-.037);

  \draw[figCobalt,line width=1.4pt]
    (4.35,.037)--(5.35,.037);
  \draw[figRed,line width=1.4pt]
    (4.35,-.037)--(5.35,-.037);

  \node[figRed,above] at (.5,.08) {$1^+$};
  \node[figCobalt,below] at (.5,-.10) {$1^-$};
  \node[figCobalt,above] at (4.85,.08) {$2^+$};
  \node[figRed,below] at (4.85,-.10) {$2^-$};

  \draw[figarrow,figRed,densely dashed]
    (1.23,.10) .. controls (2.1,.52) and (3.3,-.45)
    .. (4.13,-.08);

  \draw[figarrow,figCobalt,densely dashed]
    (1.23,-.10) .. controls (2.1,-.53) and (3.3,.42)
    .. (4.13,.08);

  \node[below,align=center] at (2.98,-.69)
    {$1^+\sim 2^-$\qquad $1^-\sim 2^+$};
\end{tikzpicture}
\caption{Crosswise slit gluing.}
\label{fig:slit-gluing}
\end{figure}
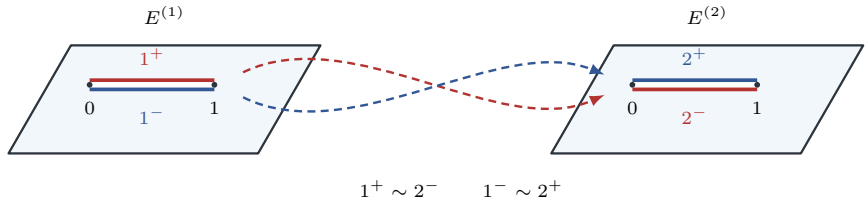

In each copy of $E$, six face corners meet at each vertex, giving total
angle $2\pi$. The crosswise gluing makes twelve corners meet at each
branch point of $X$, so the total flat angle there is $4\pi$.
Figure~\ref{fig:angle-development} develops these twelve sectors into
two consecutive turns.
\begin{figure}[H]
\centering
\begin{tikzpicture}[x=.64cm,y=.95cm,figs,>=Latex]
  \foreach \j in {0,...,11}{
    \ifodd\j
      \fill[figBlue] (\j,0) rectangle ++(1,.83);
    \else
      \fill[figGold] (\j,0) rectangle ++(1,.83);
    \fi
    \draw[figedge] (\j,0) rectangle ++(1,.83);
  }

  \draw[figframe,densely dashed] (6,-.08)--(6,1.02);

  \node at (3,1.18) {first turn: $2\pi$};
  \node at (9,1.18) {second turn: $2\pi$};

  \node[below] at (0,-.03) {$0$};
  \node[below] at (6,-.03) {$2\pi$};
  \node[below] at (12,-.03) {$4\pi$};

  \draw[figarrow] (12.35,.415)--(13.65,.415);
  \node[above,align=center] at (13,.56) {$P$};

  \foreach \j in {0,...,5}{
    \pgfmathsetmacro{\xx}{14+\j}
    \ifodd\j
      \fill[figBlue] (\xx,0) rectangle ++(1,.83);
    \else
      \fill[figGold] (\xx,0) rectangle ++(1,.83);
    \fi
    \draw[figedge] (\xx,0) rectangle ++(1,.83);
  }

  \node at (17,1.18) {one turn: $2\pi$};
  \node[below] at (14,-.03) {$0$};
  \node[below] at (20,-.03) {$2\pi$};

  \node at (6,-.69) {angular development in $X$};
  \node at (17,-.69) {in $E$};
\end{tikzpicture}
\caption{The $4\pi$ cone angle.}
\label{fig:angle-development}
\end{figure}
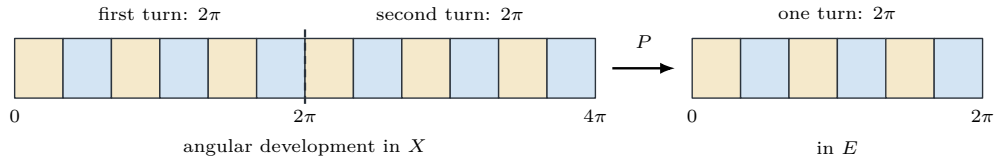

As $g(X)=2$, the uniformization theorem identifies its universal cover
with $\D$. The pullback $\pi^*\omega$ is holomorphic, and the simple
connectivity of $\D$ gives a holomorphic primitive $F$. Since $\pi$ is
an unbranched covering, the local form of $\omega$ above shows that,
near any lift $v$ of a branch point,
\begin{equation*}
F(z)-F(v)=\xi(z)^2
\end{equation*}
for a suitable holomorphic local coordinate $\xi$ centered at $v$.
Thus $v$ is a critical point of local degree two, and there are no
other critical points because $\omega$ has no other zeros.

Finally, let $\Delta$ be one of the abstract Euclidean faces, with
its characteristic map $\chi_\Delta:\Delta\to X$. Since $\Delta$
is simply connected, $\chi_\Delta$ has a lift
$\widetilde\chi_\Delta:\Delta\to\D$. On the interior of $\Delta$,
\[
d\bigl(F\circ\widetilde\chi_\Delta\bigr)=dw,
\]
so $F\circ\widetilde\chi_\Delta(w)=w+c_\Delta$ for some constant
$c_\Delta\in\C$. This identity extends to the closed face by
continuity and shows that both $\widetilde\chi_\Delta$ and the
restriction of $F$ to its image are injective. Hence $F$ maps
every closed lifted face homeomorphically onto a unit equilateral
triangle. Twelve faces meet at a lifted zero and develop onto six
planar triangles, each covered twice. The lifted
faces form a locally finite triangulation of $\D$, although images of
distinct faces under $F$ may overlap.


\end{proof}
Let $\centers$ be the set of inverse images, in the lifted faces,
of the centers of their Euclidean images.

\begin{prop}\label{prop:centers}
In the Poincar\'e metric of $\D$, the lifted faces are congruent
regular hyperbolic triangles with angles $\pi/6$, and their
hyperbolic centers are precisely the points of $\centers$.
For every $p\in\centers$,
\begin{equation}\label{eq:centerderivative}
(1-|p|^2)|F'(p)|=c_0.
\end{equation}
Moreover, there is a constant $D_0<\infty$ such that every point
of $\D$ lies within hyperbolic distance $D_0$ of a point of
$\centers$.
\end{prop}

\begin{proof}
Let $T\subset\D$ be a closed regular hyperbolic triangle with angles
$\pi/6$. Replace the eight abstract flat faces of $X$ by copies of
$T$, and glue their sides by the hyperbolic isometries with the same
endpoint identifications as in $X$. Denote the resulting surface by
$Y$. The metric is smooth across each open edge, as is seen by
reflection in the geodesic side. By Lemma~\ref{lem:surface}, each
vertex of $X$ is a simple zero of $\omega$ and hence has total flat
angle $4\pi$. Thus twelve face corners meet at each vertex. The
corresponding angle in $Y$ is $12(\pi/6)=2\pi$, so there is no cone
singularity. Hence $Y$ is a smooth compact hyperbolic surface with
the same underlying topology as $X$.

On each face choose the conformal map to the corresponding flat
triangle that matches the vertices in cyclic order. The boundary
extension theorem for Jordan domains gives a homeomorphism of the
closed triangles, and the three vertex images determine the map
uniquely. Conjugating by corresponding rotations, or by a pair of
corresponding median reflections, gives another conformal map with
the same vertex images. Uniqueness therefore shows that the face
maps intertwine these symmetries.

Let $\ell$ be the hyperbolic side length. When each side is
parametrized by arclength in the positive boundary direction, its
boundary map is described by the same increasing function
$\rho:[0,\ell]\to[0,1]$, independently of the face and side.
The median reflection interchanging the endpoints of a side gives
\[
\rho(0)=0,\qquad \rho(\ell)=1,\qquad
\rho(\ell-s)=1-\rho(s)\quad(0\leq s\leq\ell).
\]
The side identifications reverse the boundary directions: they pair
the parameters $s$ and $\ell-s$ on the hyperbolic sides, and the
parameters $r$ and $1-r$ on the flat sides. The displayed identity
therefore shows that the face maps agree on every paired side.
They consequently define a homeomorphism $\Psi:Y\to X$.

At an interior point of an edge, choose a flat coordinate on $X$ in
which the edge is a straight segment. Schwarz reflection across
this segment and the corresponding geodesic in $Y$ extends each of
the two face maps holomorphically across the edge. The extensions
agree by the identity theorem, since their values agree on the
edge. Thus $\Psi$ is holomorphic away from the vertices. If $y$ is
a vertex, choose an ordinary holomorphic coordinate $\xi$ on $X$
centered at $\Psi(y)$. The function $\xi\circ\Psi$ is holomorphic
on a punctured neighborhood of $y$ and tends to $0$ at $y$.
The removable singularity theorem extends it holomorphically across
$y$. Hence $\Psi$ is a holomorphic homeomorphism, and therefore a
conformal equivalence.

By uniqueness of the metric of curvature $-1$ in a given conformal
class on a compact surface of genus two, $\Psi$ identifies the
hyperbolic metric on $Y$ with the Poincar\'e metric on $X$.
The universal covering $\pi:\D\to X$ is a local isometry for the
Poincar\'e metrics, so every lifted face is an isometric copy of $T$.

The equivariance under rotations through $2\pi/3$ shows that $F$
maps the hyperbolic center of each lifted face to the center of its
Euclidean image. This identifies the hyperbolic face centers with
$\centers$.
Figure~\ref{fig:developing} shows the local development at a vertex.
\begin{figure}[t]
\centering
\begin{tikzpicture}[x=2.10cm,y=2.10cm,figs,>=Latex]
  \pgfmathsetmacro{\rad}{sqrt(sqrt(3)-1)}
  \pgfmathsetmacro{\geoc}
    {(1+\rad*\rad)/(2*\rad*cos(15))}
  \pgfmathsetmacro{\geor}{sqrt(\geoc*\geoc-1)}
  \pgfmathsetmacro{\geoa}
    {atan2(\rad*sin(15),\geoc-\rad*cos(15))}

  \begin{scope}
    \fill[white] (0,0) circle (1);

    \foreach \k in {0,...,11}{
      \pgfmathsetmacro{\a}{30*\k}
      \pgfmathsetmacro{\m}{\a+15}

      \ifnum\k=1
        \def\filltone{figRose}
      \else
        \ifnum\k=7
          \def\filltone{figRose}
        \else
          \ifodd\k
            \def\filltone{figBlue}
          \else
            \def\filltone{figGreen}
          \fi
        \fi
      \fi

      \fill[\filltone]
        (0,0)--(\a:\rad)
        arc[start angle={\m+180+\geoa},
            end angle={\m+180-\geoa},
            radius=\geor]
        --cycle;

      \draw[figedge]
        (\a:\rad)
        arc[start angle={\m+180+\geoa},
            end angle={\m+180-\geoa},
            radius=\geor];

      \draw[figedge] (0,0)--(\a:\rad);
    }

    \draw[figframe] (0,0) circle (1);
    \fill[figInk] (0,0) circle (.023);

    \node[above] at (0,1.18)
      {$\D$ (near a lifted zero $v$)};

    \node[below,align=center] at (0,-1.18)
      {$12$ hyperbolic triangles\\
       each angle at $v$ is $\pi/6$};
  \end{scope}

  \draw[figarrow] (1.33,0)--(2.55,0);
  \node[above] at (1.94,.09) {$F$};

  \begin{scope}[shift={(4.0,0)}]
    \foreach \k in {0,...,5}{
      \pgfmathsetmacro{\a}{60*\k}
      \pgfmathsetmacro{\b}{\a+60}

      \ifnum\k=1
        \def\filltone{figRose}
      \else
        \ifodd\k
          \def\filltone{figBlue}
        \else
          \def\filltone{figGreen}
        \fi
      \fi

      \fill[\filltone]
        (0,0)--(\a:1)--(\b:1)--cycle;

      \draw[figedge]
        (0,0)--(\a:1)--(\b:1)--cycle;
    }

    \fill[figInk] (0,0) circle (.023);

    \node[above] at (0,1.18)
      {$\C$ (near $F(v)$)};

    \node[below,align=center] at (0,-1.18)
      {$6$ unit equilateral triangles\\
       each has angle $\pi/3$};
  \end{scope}
\end{tikzpicture}
\caption{Local development at a critical point.}
\label{fig:developing}
\end{figure}
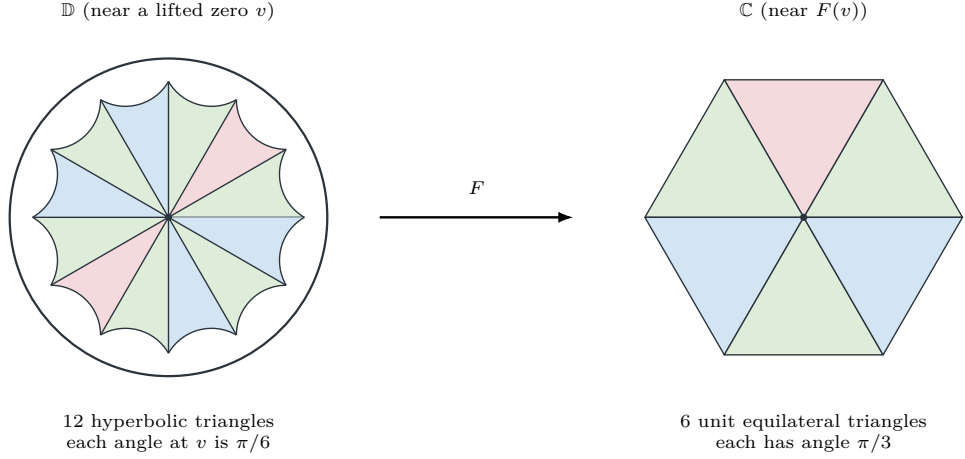

Fix $p\in\centers$, let $Q$ be its lifted face, and choose
$\phi\in\Aut(\D)$ with $\phi(0)=p$. Then $\phi^{-1}(Q)$ is a
regular hyperbolic triangle centered at $0$, and $F\circ\phi$ maps
its interior conformally onto a unit equilateral triangle, with the
vertices in cyclic order. Lemma~\ref{lem:AGtriangle} and the chain
rule give
\[
c_0=|(F\circ\phi)'(0)|
   =|F'(p)|\,|\phi'(0)|
   =(1-|p|^2)|F'(p)|,
\]
which proves \eqref{eq:centerderivative}.

Finally, since $T$ is compact, its hyperbolic diameter
$D_0:=\operatorname{diam}_{\D}(T)$ is finite. Given $z\in\D$,
choose a lifted closed face $Q$ containing $z$, and let
$p_Q\in\centers$ be its center. Congruence of the lifted faces gives
\[
d_{\D}(z,p_Q)\leq\operatorname{diam}_{\D}(Q)=D_0.
\]
\end{proof}

We next determine the Bloch radius of the developing map
$F$. Let $q:\C\to E$ be the quotient map. Choose a lift
$v_0\in\D$ of the branch point of $P$ over $0\in E$,
and normalize $F$ by $F(v_0)=0$. Since
$dF=\pi^*\omega$ and $\omega=P^*dw$, the maps
$q\circ F$ and $P\circ\pi$ have the same differential
and agree at $v_0$. Hence
\begin{equation}\label{eq:developing-projection}
q\circ F=P\circ\pi.
\end{equation}

\begin{lem}\label{lem:basecover}
The restriction
$F:\D\setminus\crit(F)\to\C\setminus L$
is a holomorphic covering map. Moreover,
$b(F)=r_0=1/\sqrt3$.
\end{lem}

\begin{proof}
Since $L=L_0\cup(1+L_0)$, we have
$q^{-1}(\{0,1\})=L$. By
\eqref{eq:developing-projection} and
Lemma~\ref{lem:surface},
\begin{equation}\label{eq:critical-lattice}
F^{-1}(L)
=\pi^{-1}\bigl(P^{-1}(\{0,1\})\bigr)
=\crit(F).
\end{equation}
In particular, $F$ maps
$\D\setminus\crit(F)$ into $\C\setminus L$
and is locally biholomorphic there.

We show that this restriction is a covering map.
Set $E^\circ=E\setminus\{0,1\}$. Both
$q:\C\setminus L\to E^\circ$ and
$P\circ\pi:\D\setminus\crit(F)\to E^\circ$
are covering maps. The latter follows by composing
the unbranched double cover
$P:X\setminus P^{-1}(\{0,1\})\to E^\circ$
with the restriction of $\pi$.

First, $F$ is surjective onto $\C\setminus L$.
Fix $z_0\in\D\setminus\crit(F)$ and let
$\gamma$ be any path in $\C\setminus L$
starting at $F(z_0)$. Lift $q\circ\gamma$
through $P\circ\pi$, starting at $z_0$,
and denote the lift by $\widetilde\gamma$.
By \eqref{eq:developing-projection},
$F\circ\widetilde\gamma$ and $\gamma$
are lifts of the same path under $q$
with the same initial point. Uniqueness of
path lifting gives
\[
F\circ\widetilde\gamma=\gamma.
\]
Since $\C\setminus L$ is path connected,
this proves surjectivity.

Now fix $w\in\C\setminus L$. Choose a
simply connected neighborhood $W$ of $q(w)$
that is evenly covered by both $q$ and
$P\circ\pi$, and let $V$ be the component
of $q^{-1}(W)$ containing $w$.

If $U$ is a component of
$(P\circ\pi)^{-1}(W)$, then $F(U)$ lies in
a single component $V_U$ of $q^{-1}(W)$.
Moreover,
\[
F|_U
=(q|_{V_U})^{-1}\circ(P\circ\pi)|_U,
\]
so $F$ maps $U$ biholomorphically onto
$V_U$. Consequently, $F^{-1}(V)$ is the
disjoint union of those components $U$
for which $V_U=V$, each mapped
biholomorphically onto $V$. Thus $V$
is evenly covered by $F$, proving the
first assertion.

It remains to compute $b(F)$.
The Voronoi cells of the triangular lattice
$L$ are regular hexagons with circumradius
$r_0=1/\sqrt3$. Their vertices are the
centers of the unit equilateral triangles
in the lattice tiling. Hence
\[
\sup_{w\in\C}\dist(w,L)=r_0.
\]
In particular, if $w_0$ is the center of
a unit equilateral triangle, then
$D(w_0,r_0)$ is disjoint from $L$.
Since this disk is simply connected and
$F$ is a covering map over $\C\setminus L$,
there exists a holomorphic inverse branch
of $F$ on $D(w_0,r_0)$. Therefore
$b(F)\geq r_0$.

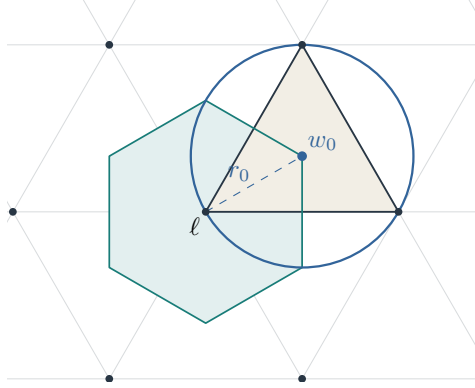
\begin{figure}[H]
\centering
\begingroup
\let\small\prooffigurefont
\begin{tikzpicture}[x=2.55cm,y=2.55cm,font=\small,line join=round]
\begin{scope}
\clip (-1.03,-.88) rectangle (1.42,1.1);
\foreach \j in {-2,...,2}{\foreach \i in {-2,...,2}{
  \coordinate (P) at ({\i+\j/2},{sqrt(3)*\j/2});
  \draw[proofInk!18,line width=.3pt] (P)--++(1,0) (P)--++(.5,{sqrt(3)/2}) (P)--++(-.5,{sqrt(3)/2});}}
\fill[proofTeal!12] (30:{1/sqrt(3)}) \foreach \a in {90,150,210,270,330}{--(\a:{1/sqrt(3)})}--cycle;
\fill[proofGold!12] (0,0)--(1,0)--(.5,{sqrt(3)/2})--cycle;
\draw[proofInk,line width=.7pt] (0,0)--(1,0)--(.5,{sqrt(3)/2})--cycle;
\draw[proofTeal,line width=.65pt] (30:{1/sqrt(3)}) \foreach \a in {90,150,210,270,330}{--(\a:{1/sqrt(3)})}--cycle;
\draw[proofBlue,line width=.85pt] (.5,{sqrt(3)/6}) circle ({1/sqrt(3)});
\draw[proofBlue,dashed] (0,0)--(.5,{sqrt(3)/6});
\foreach \j in {-2,...,2}{\foreach \i in {-2,...,2}{\fill[proofInk] ({\i+\j/2},{sqrt(3)*\j/2}) circle (.02);}}
\fill[proofBlue] (.5,{sqrt(3)/6}) circle (.025);
\node[proofBlue,above right,inner sep=2pt] at (.5,{sqrt(3)/6}) {$w_0$};
\node[below left,inner sep=2pt] at (0,0) {$\ell$};
\node[proofBlue,above left,inner sep=1pt] at (.25,.145) {$r_0$};
\end{scope}
\end{tikzpicture}

\endgroup
\caption{The lattice covering radius.}
\label{fig:lattice-radius}
\end{figure}

Conversely, let $r>r_0$ and suppose that
there is a holomorphic map
$\psi:D(w,r)\to\D$ satisfying
$F\circ\psi=\id$. By the covering-radius
property of $L$, there exists
$\ell\in L\cap D(w,r)$, as in Figure~\ref{fig:large-disk}.

\begin{figure}[H]
\centering
\begingroup
\let\small\prooffigurefont
\begin{tikzpicture}[font=\small,>=Latex]
\fill[proofBlue!4] (0,0) circle (1.45);
\draw[proofBlue,dashed,line width=.8pt] (0,0) circle (1.45);
\draw[proofBlue,->] (0,0)--(-1.04,1.01);
\node[proofBlue,above left,fill=white,inner sep=1.5pt] at (-.90,.76) {$r>r_0$};
\draw[proofRed,line width=.85pt] (0,0)--(.83,.28);
\node[proofRed,above,inner sep=2pt] at (.36,.21) {$\leq r_0$};
\fill[proofBlue] (0,0) circle (1.7pt);
\node[proofBlue,below left,inner sep=2pt] at (0,0) {$w$};
\fill[proofRed] (.83,.28) circle (1.7pt);
\node[proofRed,above right=2pt] at (.83,.28) {$\ell$};
\node[proofBlue,right=4pt] at (1.12,-.91) {$D(w,r)$};
\end{tikzpicture}

\endgroup
\caption{A disk of radius $r>r_0$.}
\label{fig:large-disk}
\end{figure}
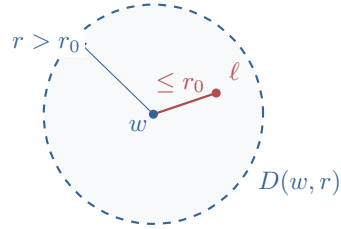

Equation
\eqref{eq:critical-lattice} implies that
$\psi(\ell)\in\crit(F)$. Differentiating
$F\circ\psi=\id$ at $\ell$, we obtain
\[
1=F'(\psi(\ell))\psi'(\ell)=0,
\]
a contradiction. Thus $b(F)\leq r_0$,
and the proof is complete.
\end{proof}

For each $p\in\centers$, let
$\varphi_p(z)=(z+p)/(1+\bar pz)$ and define
\[
g_p(z)=
\frac{F(\varphi_p(z))-F(p)}
{(1-|p|^2)F'(p)}.
\]
Then $g_p(0)=0$ and $g_p'(0)=1$.
By Proposition~\ref{prop:centers}, the preimage under
$\varphi_p$ of the face containing $p$ is a regular
hyperbolic triangle centered at $0$ with angles $\pi/6$.
The function $g_p$ maps this triangle conformally onto
a Euclidean equilateral triangle centered at the origin.
By Lemma~\ref{lem:AGtriangle} and
\eqref{eq:centerderivative}, the latter has circumradius
$\AG$. Uniqueness of the conformal triangle map,
followed by Schwarz reflection across its sides,
identifies $g_p$ with an Ahlfors--Grunsky function,
up to a simultaneous rotation of the source and target.

Since precomposition with a disk automorphism preserves
the Bloch radius, while multiplication by a constant
scales it by its modulus, Lemma~\ref{lem:basecover}
and \eqref{eq:centerderivative} give
\[
b(g_p)
=\frac{b(F)}{(1-|p|^2)|F'(p)|}
=\frac{r_0}{c_0}
=\AG.
\]
Thus every face center determines a normalized
Ahlfors--Grunsky function with Bloch radius $\AG$.
This conclusion uses only the explicit construction
and does not depend on Bloch constant conjecture.

\section{A lifted cylinder and its shear}
\label{sec:strip}

The first copy of $E$ contains the open horizontal cylinder
\[
C=\{x+iy\pmod{2\Z}:0<y<h\},
\]
viewed as a subset of $X$. It contains no zero of $\omega$,
and the integral of $\omega$ along an oriented core curve is $2$.
Fix a component $U$ of $\pi^{-1}(C)$.

\begin{prop}\label{prop:strip}
After subtracting a suitable element of $L_0$ from $F$,
the restriction $F|_U$ is a conformal map onto
\[
S_h=\{w\in\C:0<\Ima w<h\}.
\]
Its inverse $k:S_h\to U$ extends to a proper embedding
\[
\bar k:\overline{S_h}\longrightarrow\D,
\qquad
\overline{S_h}=\{w\in\C:0\le \Ima w\le h\}.
\]
The boundary curves
\[
\Gamma_0=\bar k(\R),
\qquad
\Gamma_h=\bar k(\R+ih)
\]
are disjoint in $\D$ and have the same two distinct ideal
endpoints $a,b\in\partial\D$.

Moreover, $\D\setminus(\Gamma_0\cup\Gamma_h)$ has exactly three
connected components: $U$ and two domains $\Omega_0,\Omega_h$,
where $\Gamma_0$ separates $\Omega_0$ from $U$, and $\Gamma_h$
separates $\Omega_h$ from $U$.
\end{prop}

\begin{proof}

\medskip
\noindent
\emph{The strip coordinate.}
The fundamental group of $C$ is infinite cyclic.
A loop winding $n$ times around its core has
$\omega$-period $2n$. Since $\omega$ is closed,
every null-homotopic loop in $X$ has period zero.
Thus the inclusion $C\hookrightarrow X$ induces
an injection on fundamental groups. It follows that
$
\pi|_U:U\to C
$
is the universal covering of $C$.

The standard universal covering of $C$ is the quotient map
$S_h\to C$ given by the flat coordinate modulo horizontal
translation by $2$. Consequently, there is a conformal
isomorphism
$
k:S_h\longrightarrow U
$
such that $\pi(k(w))$ is represented by $w$ in the first copy
of $E$. Since $\omega=dw$ on $C$, we have
\[
d(F\circ k)=dw.
\]
Hence
\[
F(k(w))=w+\ell
\]
for some constant $\ell\in\C$. Let $q:\C\to E$ be the quotient
map. The identity
$
q\circ F=P\circ\pi
$
implies that
\[
q(F(k(w)))=q(w),
\]
and therefore $\ell\in L_0$. Replacing $F$ by $F-\ell$
preserves the identity $q\circ F=P\circ\pi$ and gives
\[
F(k(w))=w
\]
throughout $S_h$. In particular, $F|_U$ is conformal onto $S_h$.

The translation $w\mapsto w+2$ is a deck transformation of
$S_h\to C$. It therefore corresponds to a nontrivial deck
transformation $\gamma$ of $\pi$ satisfying
\begin{equation}\label{eq:deck}
k(w+2)=\gamma k(w),
\qquad
F\circ\gamma=F+2.
\end{equation}
The second identity holds on $U$ by the first, and then on
all of $\D$ by the identity theorem.

\begin{figure}[H]
\centering
\begingroup
\let\small\prooffigurefont
\begin{tikzpicture}[font=\small,>=Latex,line join=round]
\fill[proofTeal!10] (0,0) rectangle (2.55,1.3);
\draw[proofInk,line width=.75pt] (0,0) rectangle (2.55,1.3);
\draw[proofBlue,->,line width=1pt] (0,.36)--(0,.91);
\draw[proofBlue,->,line width=1pt] (2.55,.36)--(2.55,.91);
\node at (1.275,.65) {$C$};
\node[below=3pt] at (1.275,0) {$x\sim x+2$};
\draw[proofInk,<-] (3.05,.65)--(4.6,.65);
\node[above=4pt] at (3.825,.65) {$w\bmod 2$};
\begin{scope}[shift={(5.3,0)}]
\fill[proofBlue!8] (-.25,0) rectangle (3.75,1.3);
\draw[proofInk,line width=.75pt] (0,0)--(3.5,0) (0,1.3)--(3.5,1.3);
\draw[proofInk,dashed] (-.55,0)--(0,0) (-.55,1.3)--(0,1.3) (3.5,0)--(4.05,0) (3.5,1.3)--(4.05,1.3);
\node[proofBlue] at (1.75,.65) {$S_h$};
\node[left=3pt,fill=white,inner sep=1pt] at (-.4,0) {$0$};
\node[left=3pt,fill=white,inner sep=1pt] at (-.4,1.3) {$h$};
\draw[proofTeal,->] (.7,-.36)--(2.7,-.36);
\node[below=2pt,proofTeal] at (1.7,-.36) {$w\mapsto w+2$};
\end{scope}
\end{tikzpicture}

\endgroup
\caption{The strip coordinate.}
\label{fig:strip}
\end{figure}
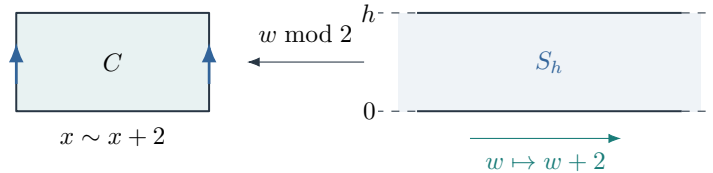

\medskip
\noindent
\emph{Extension to the boundary.}
Let
$
Q=[0,2]+i[0,h].
$
The map from the interior of $Q$ into $X$, obtained from the
first copy of $E$, extends continuously to $Q$ by taking limits
from within the cylinder. Along the slit, these limits lie on the
corresponding bank of the cut torus, while at the endpoints of
the slit they converge to the branch points of $P$. Thus the
extension is well defined and continuous on all of $Q$.

Since $Q$ is simply connected and
$
\pi:\D\to X
$
is an ordinary covering map, this extension lifts continuously
to $\D$. Choosing the lift to agree with $k$ on the interior
gives a continuous map
\[
\bar k:Q\to\D.
\]
By continuity, the identity $F(k(w))=w$ extends to
\[
F(\bar k(w))=w,
\qquad w\in Q.
\]
In particular, $\bar k$ is injective on $Q$.

The first identity in \eqref{eq:deck} extends to the vertical
sides of $Q$ by continuity. Hence we may extend $\bar k$ to
the entire closed strip by
\[
\bar k(w+2n)=\gamma^n\bar k(w),
\qquad
w\in Q,\ \ n\in\Z.
\]
These definitions agree on adjacent rectangles, so the extension
is continuous. Moreover,
\[
F(\bar k(w))=w
\]
holds throughout $\overline{S_h}$, which proves injectivity.
In particular, the images of the two boundary lines are disjoint.

\medskip
\noindent
\emph{Ideal endpoints and separation.}
Let
$
K=\bar k(Q).
$
Then $K$ is compact in $\D$, and
\[
\bar k(\overline{S_h})
=\bigcup_{n\in\Z}\gamma^nK.
\]
Since the deck group acts properly discontinuously on $\D$,
the family $\{\gamma^nK:n\in\Z\}$ is locally finite.
It follows that $\bar k$ is proper. Being continuous,
injective, and proper, $\bar k$ is a topological embedding.

The compact surface $X$ has genus two, so its universal covering
group is a cocompact torsion-free Fuchsian group. Every nontrivial
element of such a group is hyperbolic. In particular, $\gamma$
has two distinct fixed points $a,b\in\partial\D$, which we label
so that
\[
\gamma^nz\longrightarrow b
\quad(n\to+\infty),
\qquad
\gamma^nz\longrightarrow a
\quad(n\to-\infty).
\]
The convergence is uniform for $z$ in compact subsets of $\D$,
hence uniform on $K$. Together with the equivariance of $\bar k$,
this shows that both $\Gamma_0$ and $\Gamma_h$ tend to $a$ and
$b$ at their respective ends.

Properness also shows that the image of the closed strip is closed
relative to $\D$. Since every point on its boundary lines is a limit
of points of $U$, we obtain
\begin{equation}\label{eq:stripclosure}
\overline U\cap\D
=\bigcup_{n\in\Z}\gamma^nK,
\qquad
\overline U\cap\partial\D=\{a,b\}.
\end{equation}

Each of $\Gamma_0$ and $\Gamma_h$ is therefore a simple proper arc
in $\D$ whose closure in $\overline\D$ is a crosscut joining $a$
to $b$. The two crosscuts have disjoint interiors. By the Jordan
separation theorem, their union divides $\D$ into three connected
components.

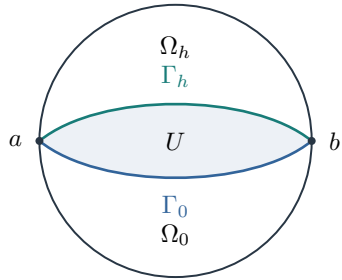
\begin{figure}[H]
\centering
\begingroup
\let\small\prooffigurefont
\begin{tikzpicture}[font=\small,line join=round]
\draw[proofInk,line width=.75pt] (0,0) circle (1.8);
\fill[proofBlue!9] (-1.8,0) .. controls (-1,.65) and (1,.65) .. (1.8,0)
.. controls (1,-.65) and (-1,-.65) .. cycle;
\draw[proofTeal,line width=1pt] (-1.8,0) .. controls (-1,.65) and (1,.65) .. (1.8,0);
\draw[proofBlue,line width=1pt] (-1.8,0) .. controls (-1,-.65) and (1,-.65) .. (1.8,0);
\fill[proofInk] (-1.8,0) circle (1.5pt) (1.8,0) circle (1.5pt);
\node[left=3pt] at (-1.8,0) {$a$}; \node[right=3pt] at (1.8,0) {$b$};
\node at (0,0) {$U$};
\node at (0,1.24) {$\Omega_h$}; \node at (0,-1.24) {$\Omega_0$};
\node[proofTeal,above=3pt] at (0,.49) {$\Gamma_h$};
\node[proofBlue,below=3pt] at (0,-.49) {$\Gamma_0$};
\end{tikzpicture}

\endgroup
\caption{The boundary crosscuts (schematic).}
\label{fig:strip-crosscuts}
\end{figure}

The component between the two curves is $U$, because $\bar k$ is
an embedding of the closed strip and
\[
\partial U\cap\D=\Gamma_0\cup\Gamma_h.
\]
The remaining components lie on the opposite sides of
$\Gamma_0$ and $\Gamma_h$ from $U$; denote them by
$\Omega_0$ and $\Omega_h$, respectively. This proves the proposition.
\end{proof}

At a boundary vertex, the selected strip occupies exactly one flat sector
of angle $\pi$. Since the local degree of $F$ there is two, this
corresponds to an angle $\pi/2$ in a regular local coordinate on $\D$.
Thus each boundary curve has a corner at such a vertex. In particular,
$\Gamma_0\cup\Gamma_h$ is not the union of all branches of the level
set of $\Ima F$. Other lifts of horizontal cylinders may lie in
$\Omega_0$ or $\Omega_h$ and will not be sheared.

Define a continuous real-valued function on $\D$ by
\begin{equation}\label{eq:u}
u=0\ \text{on }\Omega_0,\qquad
u=\Ima F\ \text{on }U,\qquad
u=h\ \text{on }\Omega_h,
\end{equation}
with the common limiting values on the boundary curves.

For real $t$, set
\[
G_t=F+tu.
\]
\begin{lem}\label{lem:shear}
The function $u$ belongs to $W^{1,\infty}_{\mathrm{loc}}(\D)$
and satisfies
$
du=\ind_U\,d\Ima F
$
almost everywhere on $\D$.
For every $t\in\R$, the map $G_t=F+tu$ is quasiregular,
with Beltrami coefficient
\begin{equation}\label{eq:mu}
\mu_t=\frac{it}{2-it}\,\ind_U\frac{\overline{F'}}{F'},
\qquad
\|\mu_t\|_\infty=\frac{|t|}{\sqrt{4+t^2}},
\qquad
J_{G_t}=|F'|^2.
\end{equation}
The quotient $\overline{F'}/F'$ is taken to be zero
at the zeros of $F'$.
\end{lem}

\begin{proof}
By Proposition~\ref{prop:strip}, the boundary curves
of $U$ consist of edges of the lifted triangulation,
possibly meeting at vertices. Since the triangulation
is locally finite, every compact subset of $\D$
meets only finitely many such edges.

The function $u$ is smooth on each of
$U$, $\Omega_0$, and $\Omega_h$, and its traces
agree across their common boundaries.
Indeed, $\Ima F$ tends to $0$ on $\Gamma_0$
and to $h$ on $\Gamma_h$. Integration by parts
on the individual faces therefore shows that
the boundary terms cancel, so the distributional
derivatives of $u$ agree with its piecewise
classical derivatives. In particular,
\[
du=\ind_U\,d\Ima F
\]
almost everywhere on $\D$.

These derivatives are locally bounded. To see
this near a vertex $v$, recall the selected sector in
Figure~\ref{fig:selected-sector}. The map $F$ has local degree two there, so
$
F'(z)=O(|z-v|)
$
as $z\to v$. Hence the piecewise gradients
of $u$ remain bounded near $v$.

\begin{figure}[H]
\centering
\begingroup
\let\small\prooffigurefont
\begin{tikzpicture}[font=\small,>=Latex,line cap=round]
\fill[proofInk!3] (0,0) circle (1.52);
\fill[proofTeal!15] (0,0)--(1.52,0) arc(0:90:1.52)--cycle;
\draw[proofInk!30,dashed] (0,0) circle (1.52);
\draw[proofInk!40,dashed] (-1.52,0)--(0,0)--(0,-1.52);
\draw[proofBlue,line width=1pt] (1.52,0)--(0,0)--(0,1.52);
\draw[proofInk!60] (.32,0) arc(0:90:.32);
\node[font=\scriptsize] at (.5,.21) {$\pi/2$};
\fill[proofInk] (0,0) circle (1.5pt);
\node[below left=2pt] at (0,0) {$v$};
\node[proofBlue,above=2pt] at (0,1.52) {$\Gamma_0$};
\node[proofBlue,right=2pt] at (1.52,0) {$\Gamma_0$};
\node[proofTeal] at (.66,1.04) {$U$};
\node[proofInk!80] at (-.82,.72) {$\Omega_0$};
\node[proofInk!80] at (.75,-.68) {$\Omega_0$};
\coordinate (Z1) at (.73,.59);
\coordinate (Z2) at (-.73,-.59);
\draw[proofInk!30,dotted] (Z2)--(Z1);
\fill[proofTeal] (Z1) circle (1.7pt);
\fill[proofGold!90!black] (Z2) circle (1.7pt);
\node[proofTeal,above left=2pt] at (Z1) {$\zeta_1$};
\node[proofGold!80!black,below right=2pt] at (Z2) {$-\zeta_1$};
\draw[proofTeal!60] (Z1)--(1.7,.95)--(2,.95);
\node[proofTeal,right=2pt] at (2,.95) {$u=\Ima F$};
\draw[proofGold!70!black] (Z2)--(-1.7,-.9)--(-2,-.9);
\node[proofGold!80!black,left=2pt] at (-2,-.9) {$u=0$};
\node[below=4pt] at (0,-1.52) {$F-F(v)=\zeta^2$};
\end{tikzpicture}

\endgroup
\caption{The selected sector at $v\in\Gamma_0$.}
\label{fig:selected-sector}
\end{figure}
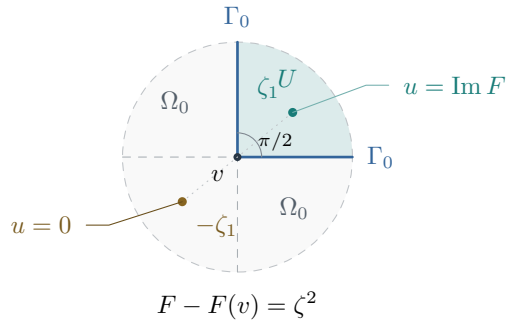

Away from the vertices, local boundedness
follows directly from the holomorphy of $F$.
Thus $u\in W^{1,\infty}_{\mathrm{loc}}(\D)$.

Since $F$ is holomorphic, we have
\[
\partial_z\Ima F=\frac{F'}{2i},
\qquad
\partial_{\bar z}\Ima F
=-\frac{\overline{F'}}{2i}.
\]
Consequently,
\[
(G_t)_z
=\left(1-\frac{it}{2}\ind_U\right)F',
\qquad
(G_t)_{\bar z}
=\frac{it}{2}\ind_U\overline{F'}
\]
almost everywhere on $\D$.
Where $F'\ne0$, taking the quotient gives
\[
\frac{(G_t)_{\bar z}}{(G_t)_z}
=\frac{it}{2-it}\,
\ind_U\frac{\overline{F'}}{F'}.
\]
Since the zeros of $F'$ are discrete, this
identity determines the Beltrami coefficient
almost everywhere. Moreover, $U$ is nonempty
and $|\overline{F'}/F'|=1$ wherever $F'\ne0$,
so
\[
\|\mu_t\|_\infty
=\left|\frac{it}{2-it}\right|
=\frac{|t|}{\sqrt{4+t^2}}<1.
\]

Finally, the expressions for the Wirtinger
derivatives give
\[
\begin{aligned}
J_{G_t}
&=|(G_t)_z|^2-|(G_t)_{\bar z}|^2\\
&=\left(
\left|1-\frac{it}{2}\ind_U\right|^2
-\frac{t^2}{4}\ind_U
\right)|F'|^2\\
&=|F'|^2,
\end{aligned}
\]
where we used that $t$ is real and
$\ind_U^2=\ind_U$.

Thus $G_t$ is a nonconstant
$W^{1,2}_{\mathrm{loc}}$ map satisfying
$
|(G_t)_{\bar z}|
\leq\|\mu_t\|_\infty |(G_t)_z|
$
almost everywhere, with
$\|\mu_t\|_\infty<1$.
Hence $G_t$ is quasiregular.
\end{proof}

\section{Conformal normalization and its first variation}
\label{sec:variation}
We solve the Beltrami equation associated with the shear
and study the first variation of the hyperbolic derivative
on the complementary domains.

\subsection{The parameter family}

Set
$\nu=\ind_U\overline{F'}/F'$ on $\D$, with the quotient
defined to be zero at the zeros of $F'$. For $|z|>1$,
define its reflection by
$
\nu^\#(z)=(z^2/\bar z^2)\overline{\nu(1/\bar z)}.
$
For $|\lambda|<1/2$, consider the Beltrami coefficient
\begin{equation}\label{eq:complexmu}
M_\lambda(z)=
\begin{cases}
\dfrac{i\lambda}{2-i\lambda}\nu(z),
& |z|<1,\\[6pt]
-\dfrac{i\lambda}{2+i\lambda}\nu^\#(z),
& |z|>1.
\end{cases}
\end{equation}
Its values on $\partial\D$ are immaterial. Since
$\|\nu\|_\infty=\|\nu^\#\|_\infty=1$, we have
\[
\|M_\lambda\|_\infty
\leq\frac{|\lambda|}{2-|\lambda|}
<\frac13.
\]
Moreover, $\lambda\mapsto M_\lambda$ is holomorphic
as an $L^\infty(\C)$-valued map.

Fix three distinct points
$\zeta_1,\zeta_2,\zeta_3\in
\partial\D\setminus\{a,b\}$.
Let $H_\lambda$ be the unique quasiconformal
homeomorphism of $\widehat{\C}$ with Beltrami
coefficient $M_\lambda$ fixing these three points.
For real $t$, we write $H_t=H_\lambda|_{\lambda=t}$.

\begin{prop}\label{prop:normalization}
For every real $|t|<1/2$, the map $H_t$ restricts
to a quasiconformal self-homeomorphism of $\D$,
and
$
\widetilde f_t=G_t\circ H_t^{-1}
$
is holomorphic on $\D$.

Moreover, there exists $p_*>2$ such that
$V=\partial_tH_t|_{t=0}$ belongs to
$W^{1,p_*}_{\mathrm{loc}}(\D)$ and satisfies
\begin{equation}\label{eq:V}
H_t=\id+tV+O(t^2)
\quad\text{locally in }W^{1,p_*},
\qquad
\dbar V=\frac{i}{2}\ind_U
\frac{\overline{F'}}{F'}.
\end{equation}
Both $H_t$ and $V$ are holomorphic on
$\Omega_0\cup\Omega_h$ and extend holomorphically
across every closed subarc of
$\partial\D\setminus\{a,b\}$. Furthermore,
\begin{equation}\label{eq:tangent}
\Rea(\bar\zeta V(\zeta))=0,
\qquad
\zeta\in\partial\D\setminus\{a,b\}.
\end{equation}
\end{prop}

\begin{proof}
We first show that $H_t$ preserves the unit disk.
Let $J(z)=1/\bar z$ denote reflection in
$\partial\D$. For real $t$, the definition
\eqref{eq:complexmu} gives
\[
M_t(z)=\frac{z^2}{\bar z^2}
\overline{M_t(1/\bar z)}
\]
almost everywhere on $\C\setminus\{0\}$.
Consequently, $J\circ H_t\circ J$ has the
same Beltrami coefficient as $H_t$.
Since $J$ fixes each $\zeta_j$, both maps
have the same normalization. Uniqueness in
Lemma~\ref{lem:AB}(1) gives
\[
H_t(1/\bar z)
=\frac{1}{\overline{H_t(z)}}.
\]
In particular, $H_t$ preserves
$\partial\D$. As $H_0=\id$ and $H_t$
depends continuously on $t$, it follows
that $H_t(\D)=\D$.

On $\D$, the coefficient $M_t$ agrees
with $\mu_t$ in \eqref{eq:mu}. Thus
Lemma~\ref{lem:shear} gives
\[
(G_t)_{\bar z}=\mu_t(G_t)_z,
\qquad
(H_t)_{\bar z}=\mu_t(H_t)_z
\]
almost everywhere. Since $G_t$ is continuous
and belongs to $W^{1,2}_{\mathrm{loc}}(\D)$,
Lemma~\ref{lem:Stoilow} implies that
$
\widetilde f_t=G_t\circ H_t^{-1}
$
is holomorphic on $\D$.

We next establish the variation formula.
The family \eqref{eq:complexmu} satisfies
$M_0=0$, and hence $H_0=\id$.
By Proposition~\ref{prop:AB}(1), there is
$p_*>2$ such that $\lambda\mapsto H_\lambda$
is holomorphic near $0$ with values in
$W^{1,p_*}$ on every fixed bounded disk.
Its Taylor expansion therefore gives
\[
H_t=\id+tV+O(t^2)
\]
locally in $W^{1,p_*}$. Differentiating
the Beltrami equation
\[
(H_t)_{\bar z}=M_t(H_t)_z
\]
at $t=0$, using $H_0=\id$ and $M_0=0$,
we obtain
\[
\dbar V
=\left.\partial_tM_t\right|_{t=0}
=\frac{i}{2}\ind_U
\frac{\overline{F'}}{F'}
\]
in the distributional sense. This proves
\eqref{eq:V}.

Since $M_\lambda=0$ on
$\Omega_0\cup\Omega_h$, Proposition~\ref{prop:AB}(2)
shows that $H_t$ and $V$ are holomorphic
on these domains. To obtain the boundary
extension, recall from \eqref{eq:stripclosure}
that
$
\overline U\cap\partial\D=\{a,b\}.
$
Thus every closed subarc
$J_0\subset\partial\D\setminus\{a,b\}$
has a neighborhood in $\widehat{\C}$
disjoint from both $\overline U$ and
its reflection. The coefficients $M_\lambda$
vanish identically on this neighborhood
for all $|\lambda|<1/2$.

By Proposition~\ref{prop:AB}(2), the normalized
solutions and their parameter derivatives
are holomorphic across $J_0$.
In particular, $V$ extends holomorphically
across every such arc.

Finally, $H_t$ preserves the unit circle, so
$
|H_t(\zeta)|^2=1
$
for $\zeta\in\partial\D$. Differentiating
at $t=0$ gives
\[
2\Rea\bigl(\overline{H_0(\zeta)}
V(\zeta)\bigr)=0.
\]
Since $H_0(\zeta)=\zeta$, this is precisely
\eqref{eq:tangent}.
\end{proof}

On $\Omega_0\cup\Omega_h$, define the
hyperbolic derivative of $H_t$ and its
first variation by
\begin{equation}\label{eq:response}
\Lambda_t(z)
=
\frac{(1-|z|^2)|H_t'(z)|}
{1-|H_t(z)|^2},
\qquad
q(z)=\Rea V'(z)
+\frac{2\Rea(\bar zV(z))}{1-|z|^2}.
\end{equation}
Here the derivatives are ordinary
holomorphic derivatives, since $H_t$
and $V$ are holomorphic on the
complementary domains.

Indeed, for every fixed
$z\in\Omega_0\cup\Omega_h$,
Proposition~\ref{prop:normalization} gives
\[
H_t(z)=z+tV(z)+O_z(t^2),
\qquad
H_t'(z)=1+tV'(z)+O_z(t^2).
\]
Consequently,
\[
\begin{aligned}
|H_t'(z)|
&=1+t\Rea V'(z)+O_z(t^2),\\
1-|H_t(z)|^2
&=1-|z|^2
-2t\Rea(\bar zV(z))+O_z(t^2).
\end{aligned}
\]
Substitution into \eqref{eq:response} yields
$
\Lambda_t(z)=1+tq(z)+O_z(t^2).
$
The expansion is locally uniform in
$z\in\Omega_0\cup\Omega_h$.

\subsection{Nontriviality of the conformal variation}

The next proposition is independent of the
triangular tessellation and the value of $\AG$.
It shows that the conformal variation cannot
be infinitesimally trivial on both complementary
domains.

\begin{prop}\label{prop:nontrivial}
Let $F$ be holomorphic on $\D$ and have at least
one critical point. Suppose that
$u\in W^{1,\infty}_{\mathrm{loc}}(\D)$ is
continuous, takes the distinct constant values
$0$ and $h$ on two domains $\Omega_0$ and
$\Omega_h$, respectively, and satisfies
$
\dbar u=F'\dbar V
$
for some continuous
$V\in W^{1,p}_{\mathrm{loc}}(\D)$, $p>2$.

Assume that $V$ is holomorphic on both
$\Omega_0$ and $\Omega_h$ and extends
holomorphically across a nonempty open arc
of $\partial\D$ adjacent to each domain.
Suppose also that
$
\Rea(\bar\zeta V(\zeta))=0
$
on these arcs. Then $V'''$ cannot vanish
identically on both $\Omega_0$ and $\Omega_h$.
\end{prop}

\begin{proof}
Suppose, to the contrary, that
$V'''$ vanishes identically on both domains.
Since $\Omega_0$ and $\Omega_h$ are connected,
there are quadratic polynomials $K_0$ and
$K_h$ such that
$
V=K_j
$
on $\Omega_j$, $j=0,h$.

We first determine the form of these
polynomials. Write
$
K_j(z)=A_j+B_jz+C_jz^2.
$
By hypothesis,
$
\Rea(\bar\zeta K_j(\zeta))=0
$
on a nonempty open arc of $\partial\D$.
Since $|\zeta|=1$, this becomes
\[
\Rea\bigl(
A_j\bar\zeta+B_j+C_j\zeta
\bigr)=0.
\]
The left-hand side is a trigonometric
polynomial. As it vanishes on an open
interval, it vanishes identically on
$\partial\D$. Comparing Fourier
coefficients gives
$
\Rea B_j=0
$
and
$
C_j=-\overline{A_j}.
$
Thus
\begin{equation*}
K_j(z)=\alpha_j+i\beta_jz
-\overline{\alpha_j}z^2,
\qquad
\alpha_j\in\C,\quad
\beta_j\in\R.
\end{equation*}
These are precisely the holomorphic
vector fields generating one-parameter
groups of disk automorphisms.

Set
$
W=V-K_0
$
and
$
\Psi=u-F'W.
$
Since $K_0$ is holomorphic,
$
\dbar W=\dbar V.
$
The assumptions on $u$ and $V$, together
with the holomorphy of $F'$, imply that
$\Psi\in W^{1,p}_{\mathrm{loc}}(\D)$
and is continuous. Moreover, the
distributional product rule gives
\[
\dbar\Psi
=\dbar u-F'\dbar W
=\dbar u-F'\dbar V
=0.
\]
By Lemma~\ref{lem:Weyl}, $\Psi$ is
holomorphic on $\D$.

On $\Omega_0$, we have $u=0$ and
$W=V-K_0=0$. Hence $\Psi$ vanishes
on $\Omega_0$, and the identity theorem
gives
$
\Psi\equiv0
$
throughout $\D$.

On $\Omega_h$, we have $u=h$ and
$W=K_h-K_0$. It follows that
\[
F'(K_h-K_0)=h
\]
on $\Omega_h$. Both sides are holomorphic
on $\D$, so the identity theorem extends
this equality to all of $\D$.

Let $v$ be a critical point of $F$.
Evaluating the identity at $v$ gives
$
0=F'(v)(K_h(v)-K_0(v))=h,
$
contrary to $h\ne0$. This contradiction
proves the proposition.
\end{proof}

For the functions $u$ and $V$ defined in
\eqref{eq:u} and \eqref{eq:V}, respectively,
we have
\[
\dbar u
=\frac{i}{2}\ind_U\overline{F'}
=F'\dbar V
\]
almost everywhere on $\D$. The zeros of $F'$
are isolated and do not affect this identity.
Together with Proposition~\ref{prop:normalization},
this verifies the hypotheses of
Proposition~\ref{prop:nontrivial}.

We also note that adding an infinitesimal disk
automorphism to $V$ does not change $q$.
Indeed, if
$
K(z)=\alpha+i\beta z-\bar\alpha z^2
$
with $\alpha\in\C$ and $\beta\in\R$, then
\[
\Rea K'(z)=-2\Rea(\bar\alpha z),
\qquad
\Rea(\bar zK(z))
=(1-|z|^2)\Rea(\bar\alpha z).
\]
It follows directly from \eqref{eq:response} that
$
\Rea K'(z)+2\Rea(\bar zK(z))/(1-|z|^2)=0.
$

\subsection{Boundary behavior and triangle centers}

We first obtain a boundary expansion for $q$.
Together with the bounded hyperbolic distance
of points in $\D$ from $\centers$, this will
allow us to find a face center at which $q$ is nonzero.

\begin{lem}\label{lem:boundary}
Let $J\subset\partial\D$ be a closed arc.
Suppose that $V$ is holomorphic in a
neighborhood of $J$ and satisfies
\eqref{eq:tangent} on $J$. Define $q$ by
\eqref{eq:response}. Then
$
\phi(\zeta)=\zeta^2V'''(\zeta)
$
is real-valued on $J$. Moreover, there exist
constants $M>0$ and $r_J\in(0,1)$ such that
\begin{equation}\label{eq:boundary}
\left|
q(r\zeta)
-\frac{(1-r^2)^2}{12}\phi(\zeta)
\right|
\leq M(1-r)^3,
\qquad
\zeta\in J,\quad r_J<r<1.
\end{equation}
\end{lem}

\begin{proof}
Write $\zeta=e^{i\theta}$ and set
$
b(\theta)=\bar\zeta V(\zeta).
$
The tangency condition \eqref{eq:tangent}
means that $b$ is purely imaginary on $J$.
Differentiating with respect to $\theta$,
we obtain
\[
\begin{aligned}
b'&=i(V'-b),\\
b''&=-\zeta V''+V'-b,\\
b'''+b'&=-i\zeta^2V''',
\end{aligned}
\]
where the derivatives of $V$ are evaluated
at $\zeta$. Since $b$ and all its
derivatives are purely imaginary, the
last identity shows that
$
\zeta^2V'''(\zeta)\in\R.
$
The second identity also gives
$
\Rea(\zeta V''(\zeta))
=\Rea V'(\zeta).
$

For $\zeta\in J$, put
$
a_\zeta(r)=\Rea(\bar\zeta V(r\zeta)).
$
The preceding identities imply that
\[
a_\zeta(1)=0,\qquad
a_\zeta''(1)=a_\zeta'(1),\qquad
a_\zeta'''(1)=\phi(\zeta).
\]
Furthermore, \eqref{eq:response} gives
\[
q(r\zeta)
=a_\zeta'(r)
+\frac{2r}{1-r^2}a_\zeta(r).
\]

Set $x=1-r$, $A=a_\zeta'(1)$, and
$B=\phi(\zeta)$. Taylor's formula yields
\[
\begin{aligned}
a_\zeta(1-x)
&=-Ax+\frac{A}{2}x^2
-\frac{B}{6}x^3+O_J(x^4),\\
a_\zeta'(1-x)
&=A-Ax+\frac{B}{2}x^2+O_J(x^3).
\end{aligned}
\]
Here the remainders are uniform for
$\zeta\in J$, since $V$ is holomorphic
in a neighborhood of the compact arc $J$
and its derivatives through order four
are bounded there.

Using
\[
\frac{2r}{1-r^2}
=\frac1x-\frac12-\frac{x}{4}+O(x^2),
\]
we obtain
\[
q(r\zeta)
=\frac{B}{3}x^2+O_J(x^3)
=\frac{\phi(\zeta)}{3}(1-r)^2
+O_J((1-r)^3).
\]
Finally,
\[
\frac{(1-r^2)^2}{12}
=\frac{(1-r)^2}{3}+O((1-r)^3).
\]
Since $\phi$ is bounded on $J$, the
last two estimates give \eqref{eq:boundary}.
\end{proof}

We now apply this expansion to the set
$\centers$ introduced in
Proposition~\ref{prop:centers}.

\begin{prop}\label{prop:center}
For the variation in
Proposition~\ref{prop:normalization},
there exist $\epsilon\in\{-1,1\}$ and
infinitely many distinct points
$p\in\centers\cap(\Omega_0\cup\Omega_h)$
such that
$
\epsilon q(p)>0.
$
\end{prop}

\begin{proof}
By Proposition~\ref{prop:nontrivial},
$V'''$ cannot vanish identically on both
complementary domains. Thus there is
$j\in\{0,h\}$ such that
$
V'''\not\equiv0
$
on $\Omega_j$.

Let $I$ be the open arc of $\partial\D$
adjacent to $\Omega_j$. By
Proposition~\ref{prop:normalization},
$V$ extends holomorphically across
every closed subarc of $I$.
The identity theorem implies that
$V'''$ cannot vanish identically on $I$.
By Lemma~\ref{lem:boundary},
$
\phi(\zeta)=\zeta^2V'''(\zeta)
$
is real-valued there. Consequently,
we may choose a closed subarc
$J\subset I$, a sign
$\epsilon\in\{-1,1\}$, and a constant
$m>0$ such that
\[
\epsilon\phi(\zeta)\geq m,
\qquad \zeta\in J.
\]
Applying \eqref{eq:boundary} and increasing
$r_J$ if necessary, we obtain
\[
\epsilon q(r\zeta)>0,
\qquad
\zeta\in J,\quad r_J<r<1.
\]

Fix $\zeta_*$ in the interior of $J$; see Figure~\ref{fig:center-collar}.

\begin{figure}[H]
\centering
\begingroup
\let\small\prooffigurefont
\begin{tikzpicture}[x=1cm,y=1cm,font=\small,
  line cap=round,line join=round,>=Latex]
  \path[use as bounding box] (-2.45,-1.83) rectangle (2.55,2.18);
  \fill[proofBlue!7]
    (-1.7,0) .. controls (-.95,.49) and (.95,.49) .. (1.7,0)
    .. controls (.95,-.49) and (-.95,-.49) .. cycle;
  \fill[proofTeal!13] (54:1.7)
    arc[start angle=54,end angle=126,radius=1.7]
    -- (126:1.17) arc[start angle=126,end angle=54,radius=1.17] -- cycle;
  \draw[proofInk!75,line width=.65pt] (0,0) circle (1.7);
  \draw[proofTeal,line width=1.4pt] (54:1.7)
    arc[start angle=54,end angle=126,radius=1.7];
  \draw[proofTeal!75,dashed,line width=.55pt] (54:1.17)
    arc[start angle=54,end angle=126,radius=1.17];
  \draw[proofTeal!55,densely dotted] (54:1.17)--(54:1.7)
    (126:1.17)--(126:1.7);
  \draw[proofBlue,line width=.75pt]
    (-1.7,0) .. controls (-.95,.49) and (.95,.49) .. (1.7,0);
  \draw[proofBlue,line width=.75pt]
    (-1.7,0) .. controls (-.95,-.49) and (.95,-.49) .. (1.7,0);
  \fill[proofInk] (-1.7,0) circle (1.15pt) (1.7,0) circle (1.15pt);
  \node[left=4pt] at (-1.7,0) {$a$};
  \node[right=4pt] at (1.7,0) {$b$};
  \node[proofBlue] at (0,0) {$U$};
  \node[proofBlue,above=2pt] at (.83,.285) {$\Gamma_h$};
  \node[proofBlue,below=2pt] at (.83,-.285) {$\Gamma_0$};
  \node at (-.10,.79) {$\Omega_h$};
  \node at (0,-.96) {$\Omega_0$};
  \node[proofTeal] at (0,1.42) {$\epsilon q>0$};
  \node[proofTeal,above left=2pt] at (119:1.7) {$J$};
  \node[above left=2pt] at (151:1.7) {$I$};
  \fill[proofRed] (0,1.7) circle (1.4pt);
  \node[above=5pt] at (0,1.7) {$\zeta_*$};
  \draw[proofTeal!65,line width=.4pt] (59:1.17)--(1.49,1.02);
  \node[proofTeal,anchor=west,inner sep=2pt] at (1.49,1.02) {$|z|=r_J$};
\end{tikzpicture}

\endgroup
\caption{A collar of constant sign.}
\label{fig:center-collar}
\end{figure}
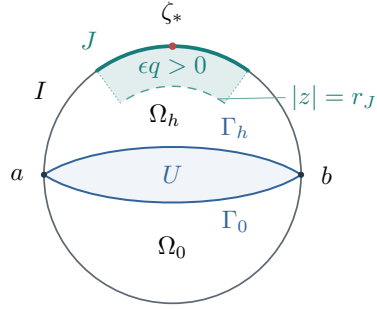

Choose a sequence $r_n\uparrow1$.
Set $z_n=r_n\zeta_*$. By
Proposition~\ref{prop:centers}, there
are points $p_n\in\centers$ such that
$
d_{\D}(z_n,p_n)\leq D_0
$
for every $n$.

Put $\eta=\tanh(D_0/2)<1$.
By the formula for the Poincar\'e distance,
\[
\left|
\frac{p_n-z_n}{1-\overline{z_n}p_n}
\right|
\leq\eta.
\]
Since
\[
|1-\overline{z_n}p_n|
\leq1-|z_n|^2+|z_n||p_n-z_n|,
\]
we obtain
\begin{equation*}
|p_n-z_n|
\leq
\frac{\eta(1-|z_n|^2)}
{1-\eta|z_n|}
\leq
\frac{2\eta}{1-\eta}(1-r_n).
\end{equation*}
In particular, $p_n\to\zeta_*$.

Since $\zeta_*$ lies in the interior
of $J$, it follows that, for all
sufficiently large $n$,
\[
\frac{p_n}{|p_n|}\in J,
\qquad
|p_n|>r_J.
\]
Moreover, $\zeta_*\notin\{a,b\}$.
By \eqref{eq:stripclosure}, the boundary
curves $\Gamma_0$ and $\Gamma_h$ have
no accumulation points on $I$.
A sufficiently small neighborhood of
$\zeta_*$ inside $\D$ therefore lies
entirely in $\Omega_j$. Hence
$p_n\in\Omega_j$ for all sufficiently
large $n$.

\begin{figure}[H]
\centering
\begingroup
\let\small\prooffigurefont
\begin{tikzpicture}[x=1cm,y=1cm,font=\small,
  line cap=round,line join=round,>=Latex]
  \path[use as bounding box] (-3.48,-1.40) rectangle (3.48,2.95);
  \def\mag{12}
  \def\yc{-9.6}
  \def\eta{.5}
  \begin{scope}
    \clip (-3.42,-1.35) rectangle (3.42,2.46);
    \fill[proofTeal!10] (0,\yc) circle (\mag);
    \fill[white] (0,\yc) circle (9);
    \draw[proofInk!75,line width=.75pt] (0,\yc) circle (\mag);
    \draw[proofTeal!65,dashed,line width=.55pt] (0,\yc) circle (9);
    \draw[proofRed!50,densely dashed,line width=.55pt] (0,-.20)--(0,2.40);
    \foreach \rr in {.90,.95,.98}{
      \pgfmathsetmacro{\ce}{(1-\eta*\eta)*\rr/(1-\eta*\eta*\rr*\rr)}
      \pgfmathsetmacro{\re}{\eta*(1-\rr*\rr)/(1-\eta*\eta*\rr*\rr)}
      \pgfmathsetmacro{\cy}{\yc+\mag*\ce}
      \pgfmathsetmacro{\rad}{\mag*\re}
      \pgfmathsetmacro{\zy}{\yc+\mag*\rr}
      \draw[proofBlue,line width=.7pt] (0,\cy) circle (\rad);
      \fill[proofRed] (0,\zy) circle (1.15pt);
    }
  \end{scope}
  \fill[proofRed] (0,2.4) circle (1.4pt);
  \node[above=5pt] at (0,2.4) {$\zeta_*$};
  \node[above=3pt] at (2.84,2.05) {$\partial\D$};
  \fill[proofGold] (.72,1.12) circle (1.4pt)
    (.38,1.77) circle (1.3pt) (.13,2.14) circle (1.1pt);
  \draw[proofInk!40,densely dotted] (0,1.20)--(.72,1.12);
  \node[proofRed,anchor=east,inner sep=2pt] at (-.10,1.28) {$z_n$};
  \node[proofGold,anchor=north west,inner sep=2pt] at (.75,1.10) {$p_n$};
  \node[proofTeal] at (2.24,1.22) {$\epsilon q>0$};
  \node[proofInk] at (0,-1.19) {$d_{\D}(z_n,p_n)\leq D_0$};
  \node[proofBlue,anchor=east,inner sep=2pt] at (-1.64,.48)
    {$B_{\D}(z_n,D_0)$};
  \draw[proofBlue!55,line width=.4pt] (-1.73,.29)--(-1.32,0);
  \node[proofTeal,below=3pt] at (2.18,-.87) {$|z|=r_J$};
\end{tikzpicture}

\endgroup
\caption{Face centers approach the boundary (schematic).}
\label{fig:center-balls}
\end{figure}
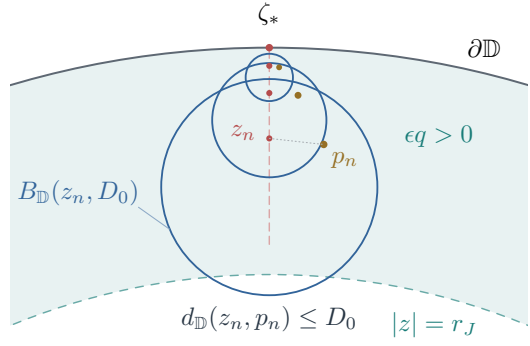

The sign estimate obtained above now gives
$
\epsilon q(p_n)>0
$
for all sufficiently large $n$.
Finally, since $p_n\to\zeta_*\in\partial\D$,
the sequence contains infinitely many
distinct points. This proves the proposition.
\end{proof}

\section{A bound for all univalent disks}
\label{sec:radius}

We bound the Bloch radius of $\widetilde f_t$ using
the triangular faces. The key observations are that
the vertices remain critical under the deformation
and that every point of a Euclidean triangle lies
within its circumradius of some vertex.

\begin{lem}\label{lem:degree}
For every real $|t|<1/2$, the critical points of
$\widetilde f_t$ are precisely the points
$H_t(v)$ with $v\in\crit(F)$, and each has
local degree two. Moreover, $G_t$ maps every
closed lifted face homeomorphically onto a
Euclidean triangle.
\end{lem}

\begin{proof}
Let $\Delta$ be a closed lifted face. Since the
boundary curves of $U$ consist of horizontal
edges of the triangulation, the interior of
$\Delta$ lies entirely in one of the regions
$U$, $\Omega_0$, or $\Omega_h$. Accordingly,
$G_t|_\Delta$ agrees with one of the maps
\[
F+t\Ima F,\qquad F,\qquad F+th.
\]
These identities extend to the closed face
by continuity. The first map is obtained by
composing $F$ with the horizontal shear
$
x+iy\mapsto x+ty+iy,
$
whose inverse is
$
x+iy\mapsto x-ty+iy.
$
Since $F$ maps each closed lifted face
homeomorphically onto a unit equilateral
triangle, the same holds for $G_t$, with
a possibly sheared image.

We next consider the critical points. Let
$v\in\crit(F)$. In a sufficiently small
neighborhood of $v$, the definition
\eqref{eq:u} gives
\[
|u(z)-u(v)|\leq|F(z)-F(v)|.
\]
Indeed, if $v$ lies outside the boundary
of $U$, the function $u$ is locally
constant. Otherwise, $v$ lies on one
of the boundary curves, where
$u(v)=\Ima F(v)\in\{0,h\}$.
On each incident sector, $u$ is either
constant or equal to $\Ima F$, so the
inequality follows in either case.

By Lemma~\ref{lem:surface}, $F$ has local
degree two at $v$. We may therefore
choose a disk $D(v,\rho)$ such that
$F(z)\ne F(v)$ for
$0<|z-v|\leq\rho$. For $|t|<1$, we have
\begin{equation}\label{eq:degree}
\begin{aligned}
|G_t(z)-G_t(v)|
&\geq |F(z)-F(v)|
     -|t|\,|u(z)-u(v)|\\
&\geq(1-|t|)|F(z)-F(v)|>0,
\qquad 0<|z-v|\leq\rho.
\end{aligned}
\end{equation}

Consider the homotopy
\[
F(z)-F(v)+st\bigl(u(z)-u(v)\bigr),
\qquad 0\leq s\leq1.
\]
By \eqref{eq:degree}, none of these maps
vanishes on $\partial D(v,\rho)$.
Their winding numbers around the origin
are therefore equal. Since $F$ has local
degree two at $v$, the winding number
of $G_t-G_t(v)$ is also two.
Furthermore, \eqref{eq:degree} shows that
$v$ is the only preimage of $G_t(v)$
in $D(v,\rho)$. Thus the local
topological degree of $G_t$ at $v$
is exactly two.

By Proposition~\ref{prop:normalization},
$
\widetilde f_t=G_t\circ H_t^{-1}
$
is holomorphic, and $H_t$ is an
orientation-preserving homeomorphism.
Hence $\widetilde f_t$ has local degree
two at $H_t(v)$, so $H_t(v)$ is a
simple critical point of $\widetilde f_t$.

It remains to exclude other critical
points. On the interior of each face,
$G_t$ is locally injective by its affine
description. The same is true at an
interior point of an edge not belonging
to $\partial U$, since the adjacent
faces have the same affine expression
for $G_t$.

Now consider an interior point of an edge
of $\partial U$. Since $F'$ does not
vanish there, we may use $w=F(z)$
as a local coordinate. The edge is
mapped to a segment of the horizontal
line $\Ima w=c$, where $c=0$ or $h$.
On one side of the edge, $G_t$ takes
the form
$
w\mapsto w+t\Ima w,
$
while on the other it takes the form
$
w\mapsto w+tc.
$
These maps agree on $\Ima w=c$.
Both preserve the imaginary coordinate,
and for each fixed imaginary part their
real components are translations.
They therefore fit together to give
a local homeomorphism across the edge.

Thus $G_t$ is locally injective away
from $\crit(F)$. The same is true of
$\widetilde f_t$ away from
$H_t(\crit(F))$. Since $\widetilde f_t$
is holomorphic, local injectivity is
equivalent to having nonzero derivative.
This proves the lemma.
\end{proof}

We shall also use the following elementary
geometric observation, which holds for
arbitrary Euclidean triangles, including
obtuse ones.

\begin{lem}\label{lem:triangle}
Let $\Delta$ be a nondegenerate closed
Euclidean triangle with circumradius $R$.
Then every $w\in\Delta$ satisfies
$
|w-v|\leq R
$
for at least one vertex $v$ of $\Delta$.
\end{lem}

\begin{proof}
Let $v_1,v_2,v_3$ be the vertices of
$\Delta$, and let $O$ be its circumcenter.
Write
$
w=\sum_{j=1}^3\lambda_jv_j,
$
where $\lambda_j\geq0$ and
$\sum_{j=1}^3\lambda_j=1$.

Since $|v_j-O|=R$ for each $j$,
expanding the squares gives
\begin{equation*}
\begin{aligned}
\sum_{j=1}^3\lambda_j|w-v_j|^2
&=\sum_{j=1}^3\lambda_j|v_j-O|^2
  -|w-O|^2\\
&=R^2-|w-O|^2
\leq R^2.
\end{aligned}
\end{equation*}

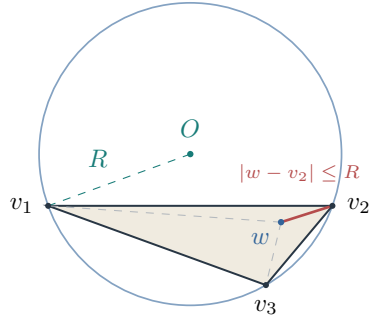
\begin{figure}[H]
\centering
\begingroup
\let\small\prooffigurefont
\begin{tikzpicture}[x=2cm,y=2cm,font=\small]
\coordinate (A) at (200:1); \coordinate (B) at (340:1); \coordinate (C) at (300:1);
\coordinate (W) at (.6,-.45);
\draw[proofBlue!60,line width=.65pt] (0,0) circle (1);
\fill[proofGold!14] (A)--(B)--(C)--cycle;
\draw[proofInk,line width=.75pt] (A)--(B)--(C)--cycle;
\draw[proofInk!35,dashed] (W)--(A) (W)--(C);
\draw[proofTeal,dashed] (0,0)--(A);
\node[proofTeal,above left=2pt] at (-.46,-.17) {$R$};
\draw[proofRed,line width=1pt] (W)--(B);
\fill[proofTeal] (0,0) circle (.021); \node[proofTeal,above=3pt] at (0,0) {$O$};
\fill[proofBlue] (W) circle (.021); \node[proofBlue,below left=1pt] at (W) {$w$};
\foreach \p/\v/\pos in {A/v_1/left,B/v_2/right,C/v_3/below}{\fill[proofInk] (\p) circle (.02);\node[\pos=2pt] at (\p) {$\v$};}
\node[proofRed,font=\scriptsize] at (.72,-.13) {$|w-v_2|\leq R$};
\end{tikzpicture}

\endgroup
\caption{A vertex within distance $R$.}
\label{fig:triangle-vertex}
\end{figure}

The left-hand side is a convex
combination of the squared distances
from $w$ to the vertices. Hence at
least one of these distances is at
most $R$, as claimed.
\end{proof}

\Needspace{9\baselineskip}
For the deformed triangles, set
\begin{equation}\label{eq:Rt}
R(t)=\frac{1}{\sqrt3}
\sqrt{1+\frac34t^2+\frac9{16}t^4}.
\end{equation}

\begin{prop}\label{prop:radius}
For every real $|t|<1/2$,
$
b(\widetilde f_t)\leq R(t).
$
\end{prop}

\begin{proof}
We first determine the circumradii of
the images of the lifted faces.

On a face contained in $U$, the map
$G_t$ is obtained from $F$ by the
horizontal shear
$
x+iy\mapsto x+ty+iy.
$
Up to translation and a half-turn,
the resulting triangle has vertices
\[
0,\qquad 1,\qquad \frac12+th+ih.
\]
Its area is $h/2$, since the shear
preserves area, and its squared side
lengths are
\[
1,\qquad
1+ht+h^2t^2,\qquad
1-ht+h^2t^2.
\]
The circumradius of a triangle with
side lengths $a,b,c$ and area $A$
is $abc/(4A)$. Thus the squared
circumradius of a sheared face is
\[
\begin{aligned}
R_{\Delta}(t)^2
&=\frac{(1+ht+h^2t^2)
         (1-ht+h^2t^2)}{4h^2}\\
&=\frac{1+h^2t^2+h^4t^4}{4h^2}\\
&=\frac13\left(
1+\frac34t^2+\frac9{16}t^4
\right).
\end{aligned}
\]
Hence $R_{\Delta}(t)=R(t)$.

\begin{figure}[H]
\centering
\begingroup
\let\small\prooffigurefont
\begin{tikzpicture}[x=3.3cm,y=3.3cm,font=\small,>=Latex]
\def\tt{.43}
\pgfmathsetmacro{\hh}{sqrt(3)/2}
\pgfmathsetmacro{\xx}{.5+\tt*\hh}
\pgfmathsetmacro{\yy}{(.5+\hh*\hh*\tt*\tt)/(2*\hh)}
\pgfmathsetmacro{\rr}{sqrt(.25+\yy*\yy)}
\coordinate (A) at (0,0); \coordinate (B) at (1,0);
\coordinate (C) at (\xx,\hh); \coordinate (C0) at (.5,\hh); \coordinate (O) at (.5,\yy);
\fill[proofBlue!8] (A)--(B)--(C)--cycle;
\draw[proofBlue!60,line width=.7pt] (O) circle (\rr);
\draw[proofInk!40,dashed] (A)--(C0)--(B);
\draw[proofInk,line width=.8pt] (A)--(B)--(C)--cycle;
\draw[proofRed,->,line width=.8pt] ($(C0)+(0,.1)$)--node[above=2pt] {$th$} ($(C)+(0,.1)$);
\draw[proofInk!30,dotted] (C0)--(.5,0);
\draw[proofTeal,line width=.8pt] (O)--(A);
\node[proofTeal,fill=white,inner sep=1pt,font=\scriptsize,rotate=37] at (.25,{\yy*.5}) {$R(t)$};
\fill[proofTeal] (O) circle (.016); \node[proofTeal,right=3pt] at (O) {$O_t$};
\foreach \p in {A,B,C}{\fill[proofInk] (\p) circle (.016);}
\node[below left=2pt] at (A) {$0$}; \node[below right=2pt] at (B) {$1$};
\node[above right=2pt,font=\scriptsize] at (C) {$\frac12+th+ih$};
\node[above left=2pt,font=\scriptsize] at (C0) {$\frac12+ih$};
\draw[<->,proofInk!65] (-.19,0)--node[left=3pt] {$h$} (-.19,\hh);
\end{tikzpicture}

\endgroup
\caption{The sheared face and its circumcircle.}
\label{fig:shear-radius}
\end{figure}
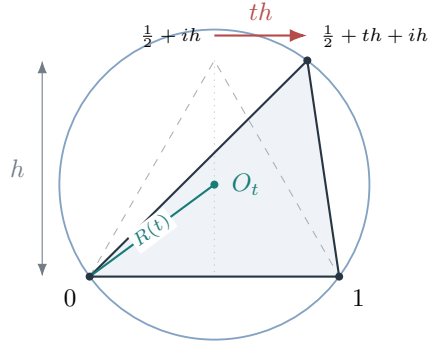

Every face outside $U$ is merely
translated, so its circumradius
remains $r_0=1/\sqrt3\leq R(t)$.
Consequently, the image of every
closed lifted face is a Euclidean
triangle of circumradius at most $R(t)$.

Suppose now that
$
\psi:D(w,r)\to\D
$
is a holomorphic inverse branch of
$\widetilde f_t$ for some $r>R(t)$.
Set $z=\psi(w)$, so that
$\widetilde f_t(z)=w$ and
$\widetilde f_t'(z)\ne0$.

Choose a closed lifted face $\Delta$
containing $H_t^{-1}(z)$. By
Lemma~\ref{lem:degree}, $G_t$ maps
$\Delta$ homeomorphically onto a
Euclidean triangle containing $w$.
Lemma~\ref{lem:triangle} gives a
vertex $v$ of $\Delta$ such that
$
|w-G_t(v)|\leq R(t)<r.
$
Since $G_t(\Delta)$ is convex, the
segment joining $w$ to $G_t(v)$
lies in $G_t(\Delta)$. Moreover, the
whole segment is contained in
$D(w,r)$.

Parametrize this segment by
\[
\sigma(s)=(1-s)w+sG_t(v),
\qquad 0\leq s\leq1,
\]
and define
\[
\alpha(s)
=
H_t\bigl((G_t|_\Delta)^{-1}(\sigma(s))\bigr),
\qquad 0\leq s\leq1.
\]
The path $\alpha$ is continuous,
and
\[
\alpha(0)=z,\qquad
\alpha(1)=H_t(v),\qquad
\widetilde f_t\circ\alpha=\sigma.
\]

For $0\leq s<1$, the point
$\alpha(s)$ is regular for
$\widetilde f_t$. Indeed, the
corresponding point of $\Delta$
is not a vertex: a segment joining
a nonvertex point of a triangle
to one of its vertices cannot
pass through another vertex
before reaching its endpoint.
Lemma~\ref{lem:degree} therefore
gives
$
\widetilde f_t'(\alpha(s))\ne0.
$

We claim that
$
\alpha(s)=\psi(\sigma(s))
$
for $0\leq s<1$. This holds at
$s=0$. The set of parameters for
which the two paths agree is closed
by continuity. It is also relatively
open in $[0,1)$, since the local
inverse theorem gives uniqueness
of lifts near every regular point
$\alpha(s)$. By connectedness,
the two paths agree throughout
$[0,1)$.

\Needspace{7.2cm}
Both paths extend continuously
to $s=1$, so
\[
\psi(G_t(v))=H_t(v).
\]

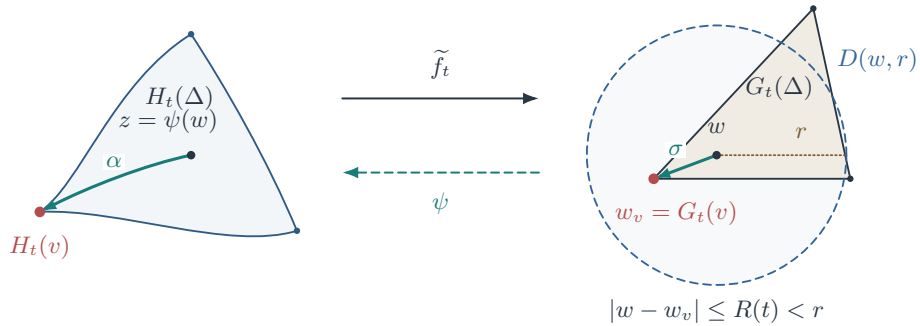
\begin{figure}[H]
\centering
\begingroup
\let\small\prooffigurefont
\begin{tikzpicture}[x=1cm,y=1cm,>=Latex,
  font=\fontsize{9}{11}\selectfont,line cap=round,line join=round]
  \path[use as bounding box] (0,-.16) rectangle (13.3,4.17);

  \coordinate (ipA) at (1.05,1.35);
  \coordinate (ipB) at (4.45,1.10);
  \coordinate (ipC) at (3.05,3.70);
  \coordinate (ipZ) at (3.05,2.10);
  \path[fill=proofBlue!6,draw=proofBlue!75!proofInk,line width=.7pt]
    (ipA) .. controls (2.10,1.40) and (3.55,.83) .. (ipB)
    .. controls (4.07,2.02) and (3.49,3.03) .. (ipC)
    .. controls (2.02,3.06) and (1.58,1.70) .. (ipA) -- cycle;
  \node[proofInk] at (2.90,2.85) {$H_t(\Delta)$};
  \draw[proofTeal,line width=1.15pt,-{Latex[length=2.2mm,width=1.45mm]}]
    (ipZ) .. controls (2.46,1.98) and (1.67,1.67) .. (ipA);
  \node[proofTeal,fill=white,inner sep=1pt] at (2.03,2.01) {$\alpha$};
  \fill[proofInk] (ipZ) circle (1.6pt);
  \node[proofInk,above=5pt,xshift=-3mm] at (ipZ) {$z=\psi(w)$};
  \fill[proofBlue!75!proofInk] (ipB) circle (1.2pt);
  \fill[proofBlue!75!proofInk] (ipC) circle (1.2pt);
  \fill[proofRed] (ipA) circle (2pt);
  \node[proofRed,below=5pt] at (ipA) {$H_t(v)$};

  \draw[proofInk,line width=.75pt,-{Latex[length=2.1mm]}]
    (5.05,2.85) -- (7.65,2.85);
  \node[proofInk,above=4pt] at (6.35,2.85) {$\widetilde f_t$};
  \draw[proofTeal,densely dashed,line width=.75pt,-{Latex[length=2.1mm]}]
    (7.65,1.87) -- (5.05,1.87);
  \node[proofTeal,below=4pt] at (6.35,1.87) {$\psi$};

  \coordinate (ipW) at (10.00,2.10);
  \coordinate (ipQ) at (9.168,1.788);
  \coordinate (ipTB) at (11.768,1.788);
  \coordinate (ipTC) at (11.278601,4.039666);
  \fill[proofBlue!4] (ipW) circle (1.716);
  \path[fill=proofGold!13,draw=proofInk,line width=.7pt]
    (ipQ) -- (ipTB) -- (ipTC) -- cycle;
  \draw[proofBlue,densely dashed,line width=.7pt]
    (ipW) circle (1.716);
  \node[proofBlue,anchor=west,fill=white,inner sep=1pt]
    at (11.55,3.36) {$D(w,r)$};
  \node[proofInk] at (10.83,2.99) {$G_t(\Delta)$};
  \draw[proofGold!75!proofInk,densely dotted,line width=.65pt]
    (ipW) -- (11.716,2.10);
  \node[proofGold!65!proofInk,above=3pt] at (11.12,2.10) {$r$};
  \draw[proofTeal,line width=1.15pt,-{Latex[length=2.2mm,width=1.45mm]}]
    (ipW) -- (ipQ);
  \node[proofTeal,fill=white,inner sep=1pt] at (9.46,2.17) {$\sigma$};
  \fill[proofInk] (ipW) circle (1.6pt);
  \node[proofInk,above=5pt] at (ipW) {$w$};
  \fill[proofInk] (ipTB) circle (1.2pt);
  \fill[proofInk] (ipTC) circle (1.2pt);
  \fill[proofRed] (ipQ) circle (2pt);
  \node[proofRed,below=5pt,xshift=3mm] at (ipQ) {$w_v=G_t(v)$};
  \node[proofInk] at (10.0,.04) {$|w-w_v|\le R(t)<r$};
\end{tikzpicture}

\endgroup
\caption{The lifted segment reaches $H_t(v)$.}
\label{fig:inverse-path}
\end{figure}

But $H_t(v)$ is a critical point
of $\widetilde f_t$ by
Lemma~\ref{lem:degree}. Differentiating
$
\widetilde f_t\circ\psi=\id
$
at $G_t(v)$ gives
\[
1
=
\widetilde f_t'(H_t(v))
\psi'(G_t(v))
=0,
\]
a contradiction.

Therefore no inverse branch exists
on a disk of radius greater than
$R(t)$, and hence
$
b(\widetilde f_t)\leq R(t).
$
\end{proof}

The argument uses only the geometry
of the face images and the local
behavior of the map. More generally,
suppose a holomorphic map admits a
locally finite decomposition into
closed triangular faces, maps each
face homeomorphically onto a Euclidean
triangle, has critical points at all
vertices, and is regular along the
open edges. Then its Bloch radius
is bounded above by the supremum
of the circumradii of the image
triangles. No Delaunay condition
or restriction on the centers of
the univalent disks is needed.

\section{Proof of the qualitative bound}
\label{sec:descent}

We now prove Theorem~\ref{thm:main}. The essential point
is to normalize the deformed map at a face center where
the first variation of the hyperbolic derivative is
nonzero. Choosing the sign of the deformation then
gives a strict improvement over $\AG$.

\begin{proof}
\emph{Normalization and comparison.}
By Proposition~\ref{prop:center}, there exists
$p\in\centers\cap(\Omega_0\cup\Omega_h)$ such that
$q(p)\ne0$. Fix such a point and put $q_0=q(p)$.

For $|t|<1/2$, let $a_t=H_t(p)$ and
$\varphi_a(z)=(z+a)/(1+\bar az)$. Define
\begin{equation}\label{eq:normalized}
f_{t,p}(z)
=
\frac{\widetilde f_t(\varphi_{a_t}(z))
      -\widetilde f_t(a_t)}
{(1-|a_t|^2)\widetilde f_t'(a_t)}.
\end{equation}
Since $p$ lies in one of the complementary domains,
$u$ is constant in a neighborhood of $p$.
Differentiating
$
\widetilde f_t\circ H_t=F+tu
$
at $p$ therefore gives
\[
\widetilde f_t'(a_t)H_t'(p)=F'(p).
\]
Here $F'(p)\ne0$, since $p$ is a face center,
and $H_t'(p)\ne0$, since $H_t$ is conformal
near $p$. Thus $f_{t,p}$ is well defined
and holomorphic on $\D$. Moreover,
$\varphi_{a_t}(0)=a_t$ and
$\varphi_{a_t}'(0)=1-|a_t|^2$, so
\[
f_{t,p}(0)=0,\qquad f_{t,p}'(0)=1.
\]

Recall that precomposition with a disk
automorphism leaves the Bloch radius unchanged,
whereas multiplication by a nonzero constant
scales it by the modulus of that constant.
Consequently,
\[
b(f_{t,p})
=
\frac{b(\widetilde f_t)}
{(1-|a_t|^2)|\widetilde f_t'(a_t)|}.
\]
On the other hand, the definition
\eqref{eq:response} of $\Lambda_t$, together
with \eqref{eq:centerderivative}, gives
\[
\begin{aligned}
(1-|a_t|^2)|\widetilde f_t'(a_t)|
&=\frac{(1-|H_t(p)|^2)|F'(p)|}
{|H_t'(p)|}\\
&=\frac{(1-|p|^2)|F'(p)|}{\Lambda_t(p)}\\
&=\frac{c_0}{\Lambda_t(p)}.
\end{aligned}
\]
Therefore Proposition~\ref{prop:radius} yields
\[
b(f_{t,p})
\leq\frac{R(t)}{c_0}\Lambda_t(p).
\]
Since $r_0/c_0=\AG$, the expression
\eqref{eq:Rt} for $R(t)$ gives
\begin{equation}\label{eq:comparison}
\frac{b(f_{t,p})}{\AG}
\leq
\sqrt{1+\frac34t^2+\frac9{16}t^4}\,
\Lambda_t(p).
\end{equation}

\medskip
\noindent
\emph{Choice of the deformation parameter.}
Since $H_t$ is holomorphic near $p$ and depends
holomorphically on the parameter near $t=0$,
both $H_t(p)$ and $H_t'(p)$ are real-analytic
functions of $t$. Moreover, $H_0=\id$.
It follows that $\Lambda_t(p)$ is
real-analytic near $t=0$, with
\[
\Lambda_0(p)=1,\qquad
\left.\frac{d}{dt}\Lambda_t(p)\right|_{t=0}
=q_0.
\]
Thus there exist $\varepsilon>0$ and
$C_p\geq0$ such that
\[
|\Lambda_t(p)-1-tq_0|\leq C_pt^2,
\qquad
\Lambda_t(p)\leq2
\]
whenever $|t|<\varepsilon$.

For $|t|\leq1$, we also have
\[
\sqrt{1+\frac34t^2+\frac9{16}t^4}
\leq1+t^2.
\]
Combining these estimates with
\eqref{eq:comparison}, we obtain
\begin{equation}\label{eq:finite}
\frac{b(f_{t,p})}{\AG}
\leq1+tq_0+(C_p+2)t^2
\end{equation}
for $|t|<\min\{\varepsilon,1/2\}$.

Choose $\delta>0$ such that
\[
\delta<
\min\left\{
\varepsilon,\frac12,
\frac{|q_0|}{2(C_p+2)}
\right\}.
\]
For $0<s<\delta$, put
$
t=-\sgn(q_0)s.
$
Then $tq_0=-|q_0|s$, and
\eqref{eq:finite} gives
\[
\begin{aligned}
\frac{b(f_{t,p})}{\AG}
&\leq1-|q_0|s+(C_p+2)s^2\\
&\leq1-\frac{|q_0|}{2}s.
\end{aligned}
\]

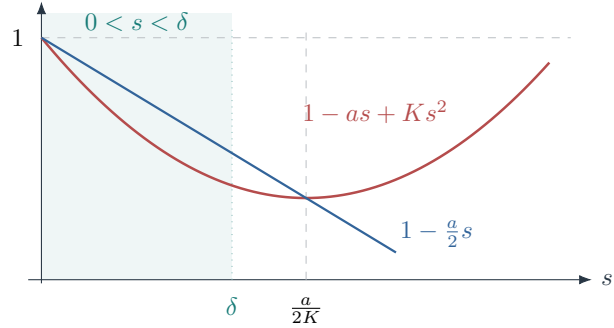
\begin{figure}[H]
\centering
\begingroup
\let\small\prooffigurefont
\begin{tikzpicture}[x=7cm,y=3.2cm,font=\small,>=Latex]
\fill[proofTeal!7] (0,0) rectangle (.36,1.1);
\draw[proofInk,->] (-.025,0)--(1.04,0) node[right] {$s$};
\draw[proofInk,->] (0,-.015)--(0,1.15);
\draw[proofInk!30,dashed] (0,1)--(1,1);
\node[left=3pt] at (0,1) {$1$};
\draw[proofTeal!50,dotted] (.36,0)--(.36,1.04);
\draw[proofInk!30,dashed] (.5,0)--(.5,1.05);
\draw[proofRed,line width=.95pt,domain=0:.96,samples=80] plot (\x,{1-2.65*\x+2.65*\x*\x});
\draw[proofBlue,line width=.85pt,domain=0:.67] plot (\x,{1-1.325*\x});
\node[proofRed,above left=1pt] at (.79,.60) {$1-as+Ks^2$};
\node[proofBlue,right=4pt] at (.64,.19) {$1-\frac a2s$};
\node[proofTeal,below=3pt] at (.36,0) {$\delta$};
\node[below=3pt] at (.5,0) {$\frac{a}{2K}$};
\node[proofTeal] at (.18,1.065) {$0<s<\delta$};
\end{tikzpicture}

\endgroup
\caption{A decreasing upper bound ($a=|q_0|$, $K=C_p+2$).}
\label{fig:first-order-descent}
\end{figure}

Thus, with $c=|q_0|/2>0$,
\[
b(f_{t,p})\leq\AG(1-cs),
\qquad 0<s<\delta.
\]

\medskip
\noindent
\emph{Convergence to the Ahlfors--Grunsky function.}
It remains to identify the limit of the
normalized functions.

By Proposition~\ref{prop:AB}(1), the maps $H_t$
converge locally uniformly to the identity
as $t\to0$. Since they are quasiconformal
self-homeomorphisms of $\D$, their inverses
also converge locally uniformly to the
identity.

The function $u$ satisfies $0\leq u\leq h$
on $\D$. Hence
$
G_t=F+tu
$
converges locally uniformly to $F$.
Together with the convergence of $H_t^{-1}$,
this implies that
$
\widetilde f_t=G_t\circ H_t^{-1}
$
converges locally uniformly to $F$.
Since each $\widetilde f_t$ is holomorphic,
the Cauchy integral formula also gives
local uniform convergence of its derivatives.

In particular, $a_t=H_t(p)\to p$ and
\[
\widetilde f_t(a_t)\longrightarrow F(p),
\qquad
\widetilde f_t'(a_t)\longrightarrow F'(p).
\]
Furthermore, $\varphi_{a_t}$ converges
locally uniformly to $\varphi_p$.
It follows from \eqref{eq:normalized} that
\[
f_{t,p}\longrightarrow
g:=
\frac{F\circ\varphi_p-F(p)}
{(1-|p|^2)F'(p)}
\]
locally uniformly on $\D$.

As established after Lemma~\ref{lem:basecover},
$g$ is a normalized Ahlfors--Grunsky function
and satisfies $b(g)=\AG$.

Finally, setting
$
f_s=f_{-\sgn(q_0)s,p}
$
for $0<s<\delta$ gives
\[
f_s\longrightarrow g
\quad\text{locally uniformly on }\D,
\qquad
b(f_s)\leq\AG(1-cs).
\]
This proves \eqref{eq:main} and completes
the proof of Theorem~\ref{thm:main}.
\end{proof}

Proposition~\ref{prop:center} provides infinitely
many possible normalization centers at which
$q$ has the same nonzero sign. The constants
$c$ and $\delta$ obtained above may depend
on the chosen center; no uniform lower bound
for $c$ is claimed.

\section{The explicit bound}
\label{sec:quantitative}

We now consider a different deformation, obtained by shearing the first
cylinder of the compact surface $X$. Unlike the deformation in
Section~\ref{sec:strip}, it acts on every lift of that cylinder. We derive
a finite-parameter estimate for the resulting normalized functions and
express its first-order term as an integral over the compact surface.
The enclosure in Lemma~\ref{lem:intervalresult} then yields the
explicit bound.

\subsection{The deformation and its normalization}

For real $t$, replace the first copy of $E$ in the construction of $X$ by
\[
E_t=\C/\bigl(2\Z+(\tau+th)\Z\bigr),
\]
leaving the second copy unchanged, and glue the two tori crosswise along
$[0,1]$ as before. Write $(X_t,\omega_t)$ for the resulting translation
surface. The shear
$A_t(x,y)=(x+ty,y)$ on the first torus and the identity on the
second agree along the banks of the slit, since the slit lies on $y=0$.
They therefore define a homeomorphism $X\to X_t$. This is a deformation
of the compact translation surface, rather than a separate deformation
of its triangular faces.

\Needspace{8.3cm}
Let $p_X$ be the center of the upward face in the first torus with flat
coordinate $w_*=3/2+ih/3$.
Figure~\ref{fig:compact-surface-shear} shows the deformation and the image of this point.
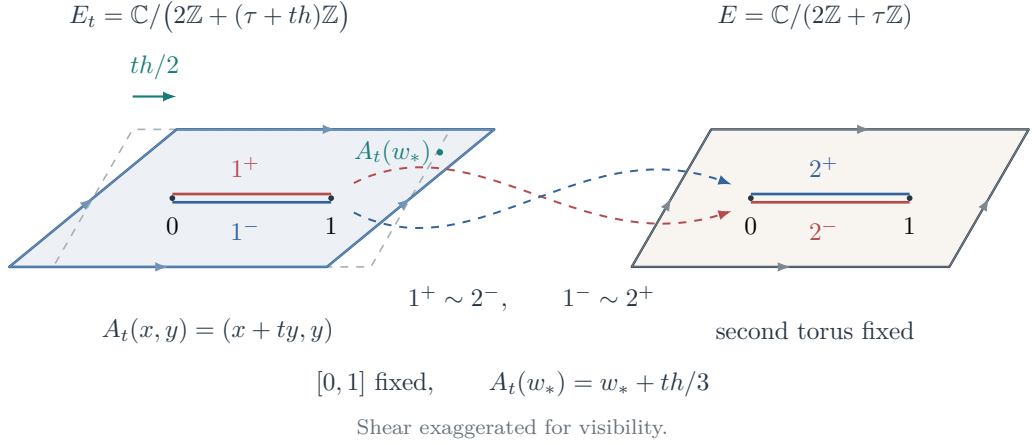
\begin{figure}[H]
\centering
\begingroup
\let\small\prooffigurefont
\begin{tikzpicture}[x=1cm,y=1cm,font=\small,>=Latex,line join=round]
\path[use as bounding box] (-.15,-2.92) rectangle (14.5,3.1);
\def\hh{0.8660254038}
\def\tt{0.64}
\node[proofInk] at (3.1,2.75) {$E_t=\mathbb C/\bigl(2\mathbb Z+(\tau+th)\mathbb Z\bigr)$};
\node[proofInk] at (11.1,2.75) {$E=\mathbb C/(2\mathbb Z+\tau\mathbb Z)$};

\begin{scope}[shift={(2.6,.35)},x=2.1cm,y=2.1cm]
  \coordinate (L0) at (-.75,-.5*\hh);
  \coordinate (R0) at (1.25,-.5*\hh);
  \coordinate (T0) at (1.75,.5*\hh);
  \coordinate (S0) at (-.25,.5*\hh);
  \coordinate (L) at (-.75-.5*\tt*\hh,-.5*\hh);
  \coordinate (R) at (1.25-.5*\tt*\hh,-.5*\hh);
  \coordinate (T) at (1.75+.5*\tt*\hh,.5*\hh);
  \coordinate (S) at (-.25+.5*\tt*\hh,.5*\hh);
  \fill[proofBlue!9] (L)--(R)--(T)--(S)--cycle;
  \draw[proofInk!40,dashed,line width=.6pt] (L0)--(R0)--(T0)--(S0)--cycle;
  \draw[proofBlue,line width=.95pt] (L)--(R)--(T)--(S)--cycle;
  \draw[proofTeal,-{Latex[length=1.8mm]},line width=.8pt]
    (-.25,.64) -- (-.25+.5*\tt*\hh,.64);
  \node[proofTeal,above=3pt] at (-.25+.25*\tt*\hh,.64) {$th/2$};
  \draw[proofBlue!70,postaction={decorate},decoration={markings,mark=at position .5 with {\arrow{Latex}}}]
    (L)--(R);
  \draw[proofBlue!70,postaction={decorate},decoration={markings,mark=at position .5 with {\arrow{Latex}}}]
    (S)--(T);
  \draw[proofBlue!70,postaction={decorate},decoration={markings,mark=at position .5 with {\arrow{Latex}}}]
    (L)--(S);
  \draw[proofBlue!70,postaction={decorate},decoration={markings,mark=at position .5 with {\arrow{Latex}}}]
    (R)--(T);
  \draw[proofRed,line width=1pt] (0,.026)--(1,.026);
  \draw[proofBlue,line width=1pt] (0,-.026)--(1,-.026);
  \fill[proofInk] (0,0) circle (.018) (1,0) circle (.018);
  \node[below=4pt] at (0,0) {$0$};
  \node[below=4pt] at (1,0) {$1$};
  \node[proofRed,above=3pt] at (.46,.026) {$1^+$};
  \node[proofBlue,below=3pt] at (.46,-.026) {$1^-$};
  \coordinate (ws) at (1.5+\tt*\hh/3,\hh/3);
  \fill[proofTeal] (ws) circle (.020);
  \node[proofTeal,anchor=east,inner sep=3pt] at (ws) {$A_t(w_*)$};
\end{scope}

\begin{scope}[shift={(10.25,.35)},x=2.1cm,y=2.1cm]
  \fill[proofGold!7] (-.75,-.5*\hh)--(1.25,-.5*\hh)--(1.75,.5*\hh)--(-.25,.5*\hh)--cycle;
  \draw[proofInk,line width=.85pt] (-.75,-.5*\hh)--(1.25,-.5*\hh)--(1.75,.5*\hh)--(-.25,.5*\hh)--cycle;
  \foreach \a/\b in {{(-.75,-.5*\hh)}/{(1.25,-.5*\hh)},{(-.25,.5*\hh)}/{(1.75,.5*\hh)},
    {(-.75,-.5*\hh)}/{(-.25,.5*\hh)},{(1.25,-.5*\hh)}/{(1.75,.5*\hh)}}{
    \draw[proofInk!60,postaction={decorate},decoration={markings,mark=at position .5 with {\arrow{Latex}}}] \a--\b;
  }
  \draw[proofBlue,line width=1pt] (0,.026)--(1,.026);
  \draw[proofRed,line width=1pt] (0,-.026)--(1,-.026);
  \fill[proofInk] (0,0) circle (.018) (1,0) circle (.018);
  \node[below=4pt] at (0,0) {$0$};
  \node[below=4pt] at (1,0) {$1$};
  \node[proofBlue,above=3pt] at (.46,.026) {$2^+$};
  \node[proofRed,below=3pt] at (.46,-.026) {$2^-$};
\end{scope}

\draw[proofRed,dashed,-{Latex[length=1.8mm]},line width=.75pt]
  (5.00,.55) .. controls (6.6,1.35) and (8.0,-.45) .. (10.0,.18);
\draw[proofBlue,dashed,-{Latex[length=1.8mm]},line width=.75pt]
  (5.00,.16) .. controls (6.6,-.65) and (8.0,1.25) .. (10.0,.53);
\node[proofInk] at (7.35,-.98) {$1^+\sim2^-,\qquad1^-\sim2^+$};
\node[proofInk] at (3.2,-1.40) {$A_t(x,y)=(x+ty,y)$};
\node[proofInk] at (11.1,-1.40) {second torus fixed};
\node[proofInk] at (7.1,-2.10) {$[0,1]\ \text{fixed},\qquad A_t(w_*)=w_*+th/3$};
\node[proofInk!70,font=\footnotesize] at (7.1,-2.70) {Shear exaggerated for visibility.};
\end{tikzpicture}

\endgroup
\caption{The compact-surface shear.}
\label{fig:compact-surface-shear}
\end{figure}

In this section choose the universal covering
$\pi:\D\to X$ so that $\pi(0)=p_X$, and choose the developing map $F$
with $F(0)=w_*$. By Proposition~\ref{prop:centers},
$|F'(0)|=c_0$. The threefold rotational symmetry of the conformal map
between the centered hyperbolic and Euclidean triangles gives
\[
F(z)=w_*+F'(0)z+O(z^4),\qquad
\frac{F'(z)}{F'(0)}=1+O(z^3)
\quad(z\to0).
\]
Indeed, the two centered triangles are invariant under rotations through
$2\pi/3$, and uniqueness of the conformal triangle map gives the stated
local equivariance. We shall use this expansion in the estimates below.

Let $\chi=\ind_{\pi^{-1}(C)}$, where $C$ is the first cylinder of $X$.
On $X$ consider the real one-form which equals $d\Ima w$ on $C$ and
vanishes on its complement. It is closed in the distributional sense:
its tangential trace is zero along each horizontal interface, and near
a zero of $\omega$ its coefficients in a regular local coordinate are
$O(|z-v|)$. In particular, there is no additional distributional term
at the vertices. Its pullback to the simply connected disk thus has a
continuous, locally Lipschitz primitive $u_C$, uniquely determined by
$u_C(0)=0$. In particular,
\[
du_C=\chi\,d\Ima F\quad\text{almost everywhere on }\D.
\]
Set $G_t^C=F+tu_C$.
On the lift containing $0$, we have $u_C=\Ima F-h/3$, so
$G_t^C=A_t\circ F-th/3$ and $G_t^C(0)=w_*$.
The computation in Lemma~\ref{lem:shear}, with
$\chi$ in place of $\ind_U$, gives a Beltrami coefficient of norm less
than one. Let $K_t:\D\to\D$ be a quasiconformal solution of that equation,
normalized by $K_t(0)=0$, and define
\begin{equation}\label{eq:compactnormal}
\mathcal F_t=G_t^C\circ K_t^{-1},\qquad
c_t=|\mathcal F_t'(0)|,\qquad
\widehat f_t=\frac{\mathcal F_t-\mathcal F_t(0)}{\mathcal F_t'(0)}.
\end{equation}
By Lemma~\ref{lem:Stoilow}, $\mathcal F_t$ is holomorphic. As $p_X$
is a regular point of $F$, the local degree of $G_t^C$ at $0$ is one.
Hence $c_t>0$, and $\widehat f_t(0)=0$, $\widehat f_t'(0)=1$.

Every lifted face is mapped by $G_t^C$ onto either a translated unit
equilateral triangle or its horizontal shear. The arguments proving
Lemma~\ref{lem:degree} and Proposition~\ref{prop:radius} apply to this
family as well: the face maps are homeomorphisms, the vertices retain
local degree two, and the inverse-path argument excludes larger
univalent disks. Using $R(t)$ from \eqref{eq:Rt}, we obtain
\begin{equation}\label{eq:compactradius}
\log\frac{b(\widehat f_t)}{\AG}
\leq \log\frac{c_0}{c_t}
+\frac12\log\left(1+\frac34t^2+\frac9{16}t^4\right).
\end{equation}
Here we used $c_0=1/(\sqrt3\,\AG)$. The bound holds for every inverse
branch; no assumption on the position of its center is needed.

\subsection{A finite-parameter comparison}

For $F'(z)\ne0$, let $Q(z)\in\mathrm{SO}(2)$ be the real rotation
corresponding to $F'(z)/|F'(z)|$, and write $Q_0=Q(0)$. Its value at a
critical point may be chosen arbitrarily. Consider the symmetric matrix
field
\begin{equation*}
\begin{aligned}
B_t&=I+tC_1+t^2C_2,\\
C_1&=\chi Q^{\mathsf T}
\begin{pmatrix}0&-1\\-1&0\end{pmatrix}Q,\qquad
C_2=\chi Q^{\mathsf T}
\begin{pmatrix}1&0\\0&0\end{pmatrix}Q.
\end{aligned}
\end{equation*}
Indeed, on $\{\chi=1\}$ the differential of $G_t^C$ is the composition
of the conformal differential of $F$ with the shear matrix
$A_t=\left(\begin{smallmatrix}1&t\\0&1\end{smallmatrix}\right)$;
on $\{\chi=0\}$ it is conformal. Consequently,
\[
B_t=(\det DG_t^C)(DG_t^C)^{-1}(DG_t^C)^{-\mathsf T}
\quad\text{almost everywhere}.
\]
This matrix is unchanged by holomorphic postcomposition, so the same
formula holds with $K_t$ in place of $G_t^C$. Since $K_t(0)=0$ and
$K_t(\D)=\D$, the function $g_t=-\log|K_t|$ is the Green function
for the corresponding divergence-form operator:
\[
-\operatorname{div}(B_t\nabla g_t)=2\pi\delta_0
\quad\text{in }\D,\qquad g_t|_{\partial\D}=0.
\]
Both the equation and the boundary condition are understood in the
standard weak and trace senses, respectively.

To compare $g_t$ with an explicit Green function having the same leading
singularity, put $a_0=1/100$, $b_0=1/50$, and define
\[
\kappa(r)=
\begin{cases}
1,&0\leq r\leq a_0,\\
2-r/a_0,&a_0<r<b_0,\\
0,&r\geq b_0.
\end{cases}
\]
\Needspace{9.1cm}
Regard complex numbers as real column vectors and set
\[
S=Q_0^{\mathsf T}
\begin{pmatrix}0&1\\0&0\end{pmatrix}Q_0,
\qquad
\phi_t(z)=z+t\kappa(|z|)Sz.
\]
The three regions of this reference map are shown in
Figure~\ref{fig:localized-reference-map}.
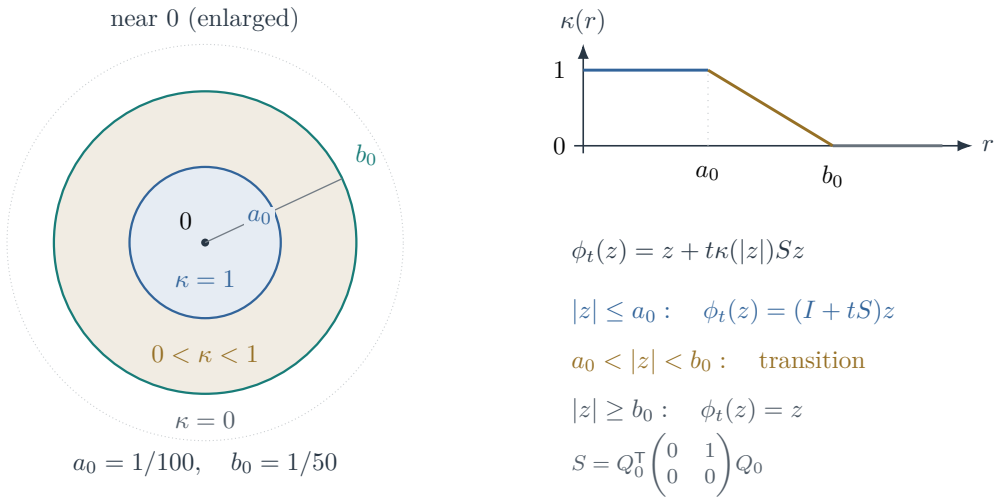
\begin{figure}[H]
\centering
\begingroup
\let\small\prooffigurefont
\begin{tikzpicture}[x=1cm,y=1cm,font=\small,>=Latex,line join=round]
\path[use as bounding box] (-3.15,-3.22) rectangle (11.75,3.10);
\fill[proofGold!13,even odd rule] (0,0) circle (2) (0,0) circle (1);
\fill[proofBlue!12] (0,0) circle (1);
\draw[proofInk!25,densely dotted] (0,0) circle (2.62);
\draw[proofTeal,line width=.85pt] (0,0) circle (2);
\draw[proofBlue,line width=.85pt] (0,0) circle (1);
\fill[proofInk] (0,0) circle (1.5pt);
\node[above left=2pt] at (0,0) {$0$};
\draw[proofInk!65,line width=.5pt] (0,0)--(25:2);
\node[proofBlue,fill=proofBlue!12,inner sep=1pt] at (25:.80) {$a_0$};
\node[proofTeal,above right=2pt] at (25:2) {$b_0$};
\node[proofBlue] at (0,-.48) {$\kappa=1$};
\node[proofGold] at (0,-1.48) {$0<\kappa<1$};
\node[proofInk!80] at (0,-2.35) {$\kappa=0$};
\node[proofInk] at (0,2.95) {near $0$ (enlarged)};
\node[proofInk] at (0,-2.92) {$a_0=1/100,\quad b_0=1/50$};

\begin{scope}[shift={(5.0,1.28)}]
  \draw[->,proofInk,line width=.6pt] (-.12,0)--(5.15,0) node[right] {$r$};
  \draw[->,proofInk,line width=.6pt] (0,-.1)--(0,1.35) node[above] {$\kappa(r)$};
  \draw[proofBlue,line width=1.1pt] (0,1)--(1.65,1);
  \draw[proofGold,line width=1.1pt] (1.65,1)--(3.3,0);
  \draw[proofInk!70,line width=1.1pt] (3.3,0)--(4.75,0);
  \draw[proofInk!30,dotted] (1.65,0)--(1.65,1);
  \node[left=3pt] at (0,1) {$1$};
  \node[left=3pt] at (0,0) {$0$};
  \node[below=4pt] at (1.65,0) {$a_0$};
  \node[below=4pt] at (3.3,0) {$b_0$};
\end{scope}
\node[anchor=west,proofInk] at (4.72,-.12) {$\phi_t(z)=z+t\kappa(|z|)Sz$};
\node[anchor=west,proofBlue] at (4.72,-.92) {$|z|\leq a_0:\quad \phi_t(z)=(I+tS)z$};
\node[anchor=west,proofGold] at (4.72,-1.57) {$a_0<|z|<b_0:\quad\text{transition}$};
\node[anchor=west,proofInk!85] at (4.72,-2.22) {$|z|\geq b_0:\quad \phi_t(z)=z$};
\node[anchor=west,proofInk!80,font=\footnotesize] at (4.72,-2.94)
  {$S=Q_0^{\mathsf T}\!\begin{pmatrix}0&1\\0&0\end{pmatrix}\!Q_0$};
\end{tikzpicture}

\endgroup
\caption{The localized reference map.}
\label{fig:localized-reference-map}
\end{figure}

The bounds $|\kappa(|z|)Sz|\leq|z|$ and
$\|D(\kappa(|z|)Sz)\|\leq3$ show that $\phi_t$ is bi-Lipschitz
for $|t|\leq1/10$. Since it is the identity near $\partial\D$, it
maps $\D$ onto itself. Define
\begin{equation*}
\begin{aligned}
\widehat B_t&=(\det D\phi_t)(D\phi_t)^{-1}(D\phi_t)^{-\mathsf T},\\
\widehat g_t&=-\log|\phi_t|,\\
I(t)&=\frac1{2\pi}\int_\D
\nabla\widehat g_t^{\mathsf T}(B_t-\widehat B_t)
\nabla\widehat g_t\,dA.
\end{aligned}
\end{equation*}
The integral is well defined: the two matrix fields agree to sufficiently
high order at the origin, where $\nabla\widehat g_t=O(|z|^{-1})$.

\begin{prop}\label{prop:greencomparison}
For $|t|\leq1/10$, let $v_t=g_t-\widehat g_t$. Then
$v_t\in W^{1,2}_0(\D)$ and
\begin{equation}\label{eq:greenidentity}
\log\frac{c_t}{c_0}
=-I(t)+\frac1{2\pi}\int_\D
\nabla v_t^{\mathsf T}B_t\nabla v_t\,dA.
\end{equation}
\end{prop}

\begin{proof}
\emph{The singularity at the origin.}
Near $0$, the coefficient of the Beltrami equation is smooth and
$G_t^C$ is the horizontal shear of $F$. Since $\mathcal F_t$ is locally
univalent, comparison of the linear terms in
$\mathcal F_t\circ K_t=G_t^C$ gives
\[
|K_t(z)|=\frac{c_0}{c_t}|A_tQ_0z|+O(|z|^2).
\]
For $|z|<a_0$, we also have $|\phi_t(z)|=|A_tQ_0z|$.
The leading logarithmic singularities therefore cancel, and $v_t$
extends continuously to $0$ with
$v_t(0)=\log(c_t/c_0)$. Local smoothness gives finite Dirichlet energy
near $0$. Away from the origin, the quasiconformal change-of-variables
formula gives finite energy; both functions have zero boundary trace.
Thus $v_t\in W^{1,2}_0(\D)$.

\emph{The energy identity.}
Put $D_t=B_t-\widehat B_t$ and
$\mathcal R_t=\operatorname{div}(D_t\nabla\widehat g_t)$.
Subtracting the two Green equations gives
\[
-\operatorname{div}(B_t\nabla v_t)=\mathcal R_t.
\]
The expansion $F'(z)/F'(0)=1+O(z^3)$ implies that
$D_t=O((|t|+t^2)|z|^3)$ near $0$, since the reference
matrix agrees there with the constant-rotation model.
Consequently $D_t\nabla\widehat g_t=O(|z|^2)$ near $0$;
away from $0$ this vector field is square integrable. In particular,
all the following pairings are justified. Testing with $v_t$ yields
\[
\langle\mathcal R_t,v_t\rangle
=\int_\D\nabla v_t^{\mathsf T}B_t\nabla v_t\,dA.
\]

For a second test, set
$\psi_\varepsilon=\min\{g_t,-\log\varepsilon\}$ and
$\widetilde v_t=v_t\circ K_t^{-1}$. The change of variables $w=K_t(z)$
and integration in polar coordinates give
\[
\begin{aligned}
\int_\D\nabla v_t^{\mathsf T}B_t\nabla\psi_\varepsilon\,dA
&=\int_{\varepsilon<|w|<1}
\nabla\widetilde v_t\cdot\nabla(-\log|w|)\,dA(w)\\
&=\int_0^{2\pi}\widetilde v_t(\varepsilon e^{i\theta})\,d\theta
\longrightarrow2\pi v_t(0).
\end{aligned}
\]
The weak equation gives, for the same limit,
\[
2\pi v_t(0)
=-\int_\D D_t\nabla\widehat g_t\cdot\nabla g_t\,dA.
\]
The bound near $0$ justifies passage to the limit there, while
Cauchy--Schwarz applies on its complement. Finally, use
$g_t=\widehat g_t+v_t$ and
$\int_\D\nabla v_t\cdot D_t\nabla\widehat g_t\,dA
=-\int_\D\nabla v_t^{\mathsf T}B_t\nabla v_t\,dA$.
The definition of $I(t)$ then gives \eqref{eq:greenidentity}.
\end{proof}

The first-order term in $I(t)$ is the real number
\begin{equation}\label{eq:qcompact}
q_C=\frac1{2\pi}\Ima\operatorname{pv}
\int_\D\frac{\chi(z)\overline{F'(z)}}{F'(z)z^2}\,dA(z).
\end{equation}
The principal value is taken by deleting $\{|z|<\varepsilon\}$ and
letting $\varepsilon\downarrow0$. At the origin the leading term of the
integrand has zero angular mean, and the expansion of $F'$ above
ensures convergence of the remaining contribution.

\begin{prop}\label{prop:explicitremainder}
For $0<|t|\leq10^{-4}$,
\[
|I(t)-tq_C|<9t^2.
\]
\end{prop}

\begin{proof}
\emph{The disk around the base point.}
The centered hyperbolic triangle contains $\{|z|\leq1/5\}$.
Indeed, in the disk model its sides lie on circles whose distance from the
origin is $C-\sqrt{C^2-1}>1/5$, where
$C^2=(5+3\sqrt3)/2$. Since its flat image is a unit equilateral
triangle centered at $w_*$,
$|F(z)-w_*|\leq r_0$ on this disk. Cauchy's estimate gives
\[
\left|\frac{F'(z)}{F'(0)}\right|
\leq\frac{20r_0}{c_0}=20\AG<12,
\qquad |z|\leq\frac1{10}.
\]
Only powers divisible by three occur in the Taylor series of
$F'/F'(0)$. Another application of Cauchy's estimate therefore gives,
for $r=|z|\leq b_0$,
\[
\left|\frac{F'(z)}{F'(0)}-1\right|
\leq\frac{12000r^3}{1-1000r^3}<12500r^3.
\]
This disk lies in $\{\chi=1\}$. Let $C_j^0$ denote the matrices obtained
from $C_j$ by setting $\chi=1$ and $Q=Q_0$. The preceding bound implies
$\|C_j-C_j^0\|<50000r^3$ for $j=1,2$.

Put $p_t=r\nabla\widehat g_t$. Direct differentiation of
$\widehat g_t$ gives, for $|t|\leq1/10$,
\[
|p_t|\leq\frac95,
\qquad |p_t-p_0|\leq8|t|.
\]
For example, the denominator in the formula for $p_t$ is bounded below
by $(1-|t|)^2$, while the numerator and its difference from the value
at $t=0$ are bounded by $(1+3|t|)(1+|t|)$ and $5|t|+4t^2$,
respectively. Hence, for any real symmetric matrix $M$,
\[
|p_t^{\mathsf T}Mp_t-p_0^{\mathsf T}Mp_0|
\leq23\|M\||t|.
\]
It follows that the contribution of the variable rotation to
$I(t)-tq_C$, over $r<b_0$, has modulus at most
\[
27\cdot50000\int_0^{b_0}r^2\,dr\,t^2
=\frac{18}{5}t^2.
\]

\emph{The transition annulus.}
Rotate coordinates so that
$S=\left(\begin{smallmatrix}0&1\\0&0\end{smallmatrix}\right)$.
Write $c=\cos\theta$, $s=\sin\theta$,
$d=r\kappa'(r)$, $\alpha=dcs$, and
$\beta=\kappa+ds^2$. Differentiating $\phi_t$ yields
\begin{equation}\label{eq:annulusmatrix}
\begin{aligned}
L:=D(\kappa Sz)
&=\begin{pmatrix}\alpha&\beta\\0&0\end{pmatrix},\\
\widehat B_t
&=\begin{pmatrix}
\dfrac{1+t^2\beta^2}{1+t\alpha}&-t\beta\\
-t\beta&1+t\alpha
\end{pmatrix},\\
p_t
&=-\frac{(I+tL)^{\mathsf T}
\bigl((c,s)^{\mathsf T}+t\kappa(s,0)^{\mathsf T}\bigr)}
{1+2t\kappa cs+t^2\kappa^2s^2}.
\end{aligned}
\end{equation}
Consider
\[
\mathcal J(t;r,\theta)
=p_t^{\mathsf T}\bigl(I+tC_1^0+t^2C_2^0-\widehat B_t\bigr)p_t.
\]
Its first-order coefficient has zero angular mean, and its second-order
coefficient has angular mean
\[
\frac{d^2+6d\kappa-4d+4\kappa^2-8\kappa+4}{8}.
\]
On $a_0<r<b_0$, substituting
$\kappa=2-r/a_0$ and $d=-r/a_0$ and integrating against $dr/r$
gives $1/16+(\log2)/2$.

For the remaining terms, regard $t$ as complex. On $|t|=1/10$ the
denominators in \eqref{eq:annulusmatrix} have moduli at least
$9/10$ and $89/100$, and $\|L\|\leq3$. In particular,
\[
|\mathcal J(t;r,\theta)|
\leq\left(\frac{143}{89}\right)^2
\left(\frac{111}{100}+\frac{169}{90}\right)<8.
\] Cauchy's estimate bounds the terms
of degree at least three by
$8000|t|^3/(1-10|t|)<8010|t|^3$ for
$|t|\leq10^{-4}$. Their total contribution, after radial integration,
is at most $8010(\log2)|t|^3$.

\emph{The exterior region.}
For $r\geq b_0$, the reference map is the identity.
The term quadratic in $t$ is nonnegative and at most
$t^2\log50$. For $r<a_0$, the reference matrix agrees with the
constant-rotation model. The first-order reference contribution has
zero angular mean; the remaining first-order term is precisely
$q_C$ in \eqref{eq:qcompact}. Combining the estimates gives
\begin{equation*}
|I(t)-tq_C|
\leq\left(\frac{18}{5}+\frac1{16}+\log50
+\left(\frac12+\frac{801}{1000}\right)\log2\right)t^2
<9t^2.
\end{equation*}
\href{https://github.com/zhangteng2000/Bloch-constant/blob/95d6d1056b7413a192a57457eab7d0d27aba7c23/bloch_quantitative_certificates/bloch_shear_refined_energy_check.py\#L20-L68}{The coefficient in parentheses} is \href{https://github.com/zhangteng2000/Bloch-constant/blob/95d6d1056b7413a192a57457eab7d0d27aba7c23/bloch_quantitative_certificates/bloch_shear_refined_energy_check_result.json\#L2-L11}{less than $8.477$}.
\end{proof}

The energy term in \eqref{eq:greenidentity} is nonnegative.
Combining that identity with \eqref{eq:compactradius} and
Proposition~\ref{prop:explicitremainder}, and using
$\frac12\log(1+3t^2/4+9t^4/16)\leq t^2$ for $|t|\leq1$,
we obtain
\begin{equation}\label{eq:finitebound}
\log\frac{b(\widehat f_t)}{\AG}
\leq tq_C+10t^2,
\qquad 0<|t|\leq10^{-4}.
\end{equation}

\subsection{Reduction to a finite integral}

We next express $q_C$ using differential forms on the compact surface.
This eliminates the infinitely many lifts appearing in
\eqref{eq:qcompact} and leads to finite integrals that can, in principle,
be enclosed by interval arithmetic.

Let $\Gamma$ be the deck group of $\pi$. Consider the Poincar\'e series
\begin{equation}\label{eq:poincareseries}
\mathcal Q(z)\,dz^2
=\sum_{\gamma\in\Gamma}
\left(\frac{\gamma'(z)}{\gamma(z)}\right)^2dz^2.
\end{equation}
Away from the $\Gamma$-orbit of $0$, the series converges locally
uniformly. In fact, writing $a=\gamma^{-1}(0)$, a summand on a fixed
compact set is $O((1-|a|^2)^2)=O(e^{-2d_\D(0,a)})$,
whereas the number of orbit points at hyperbolic distance at most $R$
is $O(e^R)$. Thus \eqref{eq:poincareseries} defines a
$\Gamma$-invariant meromorphic quadratic differential on $\D$ and
descends to $X$. Its only pole on $X$ is a double pole at $p_X$.
In the flat coordinate $w$, its principal part is
$dw^2/(w-w_*)^2$, with no simple-pole term. The latter assertion
follows from $F(z)=w_*+F'(0)z+O(z^4)$. Throughout this subsection, principal values at the poles are taken
by deleting centered disks in the local flat coordinate.
This agrees with the convention in \eqref{eq:qcompact}:
the expansion $w-w_*=F'(0)z+O(z^4)$ shows that changing
between the two cutoffs contributes $o(1)$.

Let $P_1$ be the horizontal period of $\mathcal Q/\omega$ on the first
cylinder, with orientation in the positive real direction. The period is independent of the height of the cycle, since
the pole of $\mathcal Q/\omega$ has zero residue.
Unfolding the sum \eqref{eq:poincareseries} in \eqref{eq:qcompact}
gives
\begin{equation}\label{eq:periodidentity}
q_C=\frac{h}{2\pi}\Ima P_1.
\end{equation}
One must account for the centered principal value at $p_X$.
In the coordinate $w-w_*=x+iy$, integration of $(x+iy)^{-2}$
over the chord $-\sqrt{\varepsilon^2-y^2}<x<\sqrt{\varepsilon^2-y^2}$
gives
\[
-\frac{2\sqrt{\varepsilon^2-y^2}}{\varepsilon^2}.
\]
Integrating this expression over $0<|y|<\varepsilon$ gives $-\pi$.
Here the integration is performed first in $x$ and then in $y$.
The centered two-dimensional principal value of $(x+iy)^{-2}$
over the disk is zero.
Consequently, deleting the disk after taking the full horizontal
period contributes the real correction $+\pi$: the complex area
integral is $hP_1+\pi$. Since only its imaginary part occurs in
\eqref{eq:periodidentity}, this correction does not affect $q_C$.
Away from the pole the unfolding is absolutely convergent; the flat
density near a branch point has at worst an integrable
$O(|w-w_v|^{-1})$ singularity.

Let $\iota$ be the involution of $X$ exchanging its two sheets over $E$.
The anti-invariant meromorphic one-form
\[
\mathcal A^-:=\frac{\mathcal Q-\iota^*\mathcal Q}{2\omega}
\]
is regular at the zeros of $\omega$. Indeed, in a local coordinate
with $\iota(\xi)=-\xi$ and $\omega=2\xi\,d\xi$, the coefficient of
the numerator is odd in $\xi$ and therefore divisible by $\xi$.
The only poles of $\mathcal A^-$ occur at $p_X$ and $\iota p_X$,
with opposite principal parts
$\frac12\,dw/(w-w_*)^2$ and
$-\frac12\,dw/(w-w_*)^2$, respectively.

The anti-invariant space of holomorphic one-forms on $X$ is
one-dimensional. Let $\eta$ generate it, and choose an anti-invariant
meromorphic one-form $\mathcal A_0$ with these principal parts.
Then $\mathcal A^-=\mathcal A_0+c\eta$ for some $c\in\C$.
To determine $c$, write the hyperbolic metric of $X$ as
$\rho_X(z)|dz|$ and set $\rho_\D(z)=2/(1-|z|^2)$.
The Poincar\'e series is orthogonal, in the principal-value sense,
to every holomorphic quadratic differential $\varphi$ on $X$:
\[
\operatorname{pv}\int_X
\frac{\mathcal Q\,\overline\varphi}{\rho_X^2}\,dA=0.
\]
Indeed, unfolding reduces the integral to
\[
\operatorname{pv}\int_\D
\frac{\overline{\varphi(z)}}{z^2\rho_\D(z)^2}\,dA(z),
\]
whose angular integral vanishes on every annulus centered at $0$.
The integral is absolutely convergent near $\partial\D$, since
$|\varphi|/\rho_\D^2$ is bounded. The same orthogonality holds
for $\iota^*\mathcal Q$.

Put $m=|\omega|_X^2$ and define the weighted pairing
\[
\langle\alpha,\beta\rangle_m
=\frac i2\operatorname{pv}\int_X
m\,\alpha\wedge\overline\beta.
\]
Applying orthogonality with the holomorphic quadratic differential
$\varphi=\omega\eta$ gives
\[
c=-\frac{\langle\mathcal A_0,\eta\rangle_m}
{\langle\eta,\eta\rangle_m}.
\]
The denominator is positive. In addition, the sum of the imaginary
horizontal periods of $\mathcal Q/\omega$ on the two cylinders
vanishes: replacing $\chi$ by $1$ in \eqref{eq:qcompact} gives zero,
as is seen from the threefold symmetry
$F'(e^{2\pi i/3}z)=F'(z)$. Thus the imaginary period of
$\mathcal A^-$ on the first cylinder equals $\Ima P_1$.

We make the projection explicit on an elliptic model of $X$.
All Weierstrass functions below have period lattice $L_0$.
Set $a_*=1/2$ and
\[
f_E(w)=\wp(w-a_*)-\wp(a_*).
\]
The compactification of
\[
v^2=f_E(w)
\]
is a double cover of $E$, branched at $0$ and $1$.
It gives the same marked cover as the slit construction. One way
to verify the marking is to deform the lattice through
$\tau_s=s/2+ih$, $0\leq s\leq1$, and transport the cycles
\[
a_s(x)=x+\tau_s/2\quad(0\leq x\leq2),\qquad
b_s(r)=3/2+r\tau_s\quad(0\leq r\leq1).
\]
These cycles avoid $0$, $1$, and $a_*$ modulo the corresponding
period lattice. At $s=0$, the function
$\wp_s(w-a_*)-\wp_s(a_*)$ is negative real on both cycles, so its
square root has trivial monodromy there. This monodromy remains
unchanged during the deformation, while a small loop around either
simple zero reverses the sign. These are precisely the monodromies
of the slit double cover.

On this model $\omega=dw$ and $\eta=dw/v$.
Choose the value $v_*=v(w_*)$ on the first sheet and set
$v_*'=\wp'(w_*-a_*)/(2v_*)$. Define
\begin{equation}\label{eq:ellipticR}
R_*(w)=\frac12\left[
v_*\wp(w-w_*)
+v_*'\bigl(\zeta_{L_0}(w-w_*)-\zeta_{L_0}(w-a_*)\bigr)
\right].
\end{equation}
The differential $\mathcal A_0=R_*(w)\,dw/v(w)$ has the required
principal parts. At $w=w_*$, the simple-pole terms from $1/v$ and
the zeta function cancel; on the other sheet both principal parts
change sign. At $w=a_*$, the apparent pole of $R_*$ is canceled by the
zero of $dw/v$.

The hyperbolic metric, and hence $m$, is invariant under $\iota$ and
may be regarded as a function of $w\in E$. Each weighted pairing is twice the corresponding integral over $E$,
so the factor of two cancels in their ratio. Thus, writing
$D_m=\langle\eta,\eta\rangle_m/2$ and
$M_*=-c$, we obtain the finite formulas
\begin{equation}\label{eq:finiteprojection}
\begin{gathered}
D_m=\int_E\frac{m(w)}{|f_E(w)|}\,dA>0,\qquad
M_*=\frac{\displaystyle\operatorname{pv}\int_E
\frac{m(w)R_*(w)}{|f_E(w)|}\,dA}{D_m},\\
q_C=\frac{h}{2\pi}\Ima
\int_{a_1}\frac{R_*(w)-M_*}{v(w)}\,dw.
\end{gathered}
\end{equation}
Here $a_1$ is a positively oriented horizontal cycle in the first sheet,
and $v$ is continued along that cycle. The expression is unaffected by
reversing the initial choice of square root. In the flat coordinate,
$m(w)$ is the inverse square of the hyperbolic density.

\subsection{The numerical enclosure}

We describe the evaluation of \eqref{eq:finiteprojection}
used to obtain Lemma~\ref{lem:intervalresult}.
The implementation and computational records are available in the \href{https://github.com/zhangteng2000/Bloch-constant/blob/95d6d1056b7413a192a57457eab7d0d27aba7c23/README.txt}{repository} (\href{https://github.com/zhangteng2000/Bloch-constant/blob/95d6d1056b7413a192a57457eab7d0d27aba7c23/bloch_quantitative_certificates/outer_run.log}{outer}, \href{https://github.com/zhangteng2000/Bloch-constant/blob/95d6d1056b7413a192a57457eab7d0d27aba7c23/bloch_quantitative_certificates/inner_run.log}{inner}, and \href{https://github.com/zhangteng2000/Bloch-constant/blob/95d6d1056b7413a192a57457eab7d0d27aba7c23/bloch_quantitative_certificates/period_run.log}{period} logs).

Let ${}_2F_1$ be the Gauss hypergeometric function, and set
\begin{equation*}
\begin{aligned}
W(\xi)&=c\xi\,{}_2F_1(1/3,2/3;4/3;\xi^3),
&c&=\frac{\Gamma(2/3)}{\sqrt3\,\Gamma(4/3)\Gamma(1/3)},\\
Z(\xi)&=k\xi\,
\frac{{}_2F_1(3/4,5/12;4/3;\xi^3)}
{{}_2F_1(5/12,1/12;2/3;\xi^3)},
&k&=\frac{\Gamma(2/3)\Gamma(11/12)}
{\sqrt{1+\sqrt3}\,\Gamma(4/3)\Gamma(1/4)}.
\end{aligned}
\end{equation*}
These are the flat and hyperbolic triangle coordinates, respectively;
compare the Schwarz triangle formula in
Ahlfors--Grunsky~\cite[p.~672, Lemma]{AG}.
For the upward faces of $E$ take
$w=j+1/2+ih/3+e^{7\pi i/6}W(\xi)$, $j=0,1$;
the downward faces are obtained by $w\mapsto1-w$ modulo $L_0$.
Since the hyperbolic density in the flat coordinate is
$2|Z'|/((1-|Z|^2)|W'|)$, the measure appearing in
\eqref{eq:finiteprojection} becomes
\begin{equation}\label{eq:metricweight}
m(w)\,dA(w)
=\frac{(1-|Z(\xi)|^2)^2|W'(\xi)|^4}{4|Z'(\xi)|^2}\,dA(\xi).
\end{equation}
The choice of phase agrees with the indicated vertices and the base
point $w_*$.

For the inner disks $|\xi|\leq r_i=3/4$, put
$b(\xi)=(dw/d\xi)^2/(v(w)Z'(\xi))$.
\href{https://github.com/zhangteng2000/Bloch-constant/blob/95d6d1056b7413a192a57457eab7d0d27aba7c23/bloch_quantitative_certificates/bloch_shear_prym_inner_audit.py\#L55-L161}{Expanding $bZ^j$ and $R_*bZ^j$, $j=0,1,2$}, reduces the relevant
integrals to orthogonal monomial moments. At the double pole, the
principal part is subtracted before the series is integrated; terms
with negative powers have zero angular mean against the conjugate
holomorphic series. Suppose that two regular series have modulus
bounds $M_1,M_2$ on $|\xi|=r_c=7/8$. Their product tail beyond
degree $N=160$, including the factor $1/4$ in
\eqref{eq:metricweight}, is bounded by
\begin{equation}\label{eq:innertail}
\frac{\pi r_i^2M_1M_2}{4(N+2)}
\frac{(r_i/r_c)^{2N+2}}{1-(r_i/r_c)^2}.
\end{equation}
A certificate must supply \href{https://github.com/zhangteng2000/Bloch-constant/blob/95d6d1056b7413a192a57457eab7d0d27aba7c23/bloch_quantitative_certificates/bloch_shear_prym_inner_result.json\#L3-L60}{rigorous bounds $M_1,M_2$} on the entire
circle. The downward faces may be paired using the involution
$w\mapsto1-w$, replacing $R_*(w)$ by $R_*(w)+R_*(1-w)$.

\Needspace{9.0cm}
\href{https://github.com/zhangteng2000/Bloch-constant/blob/95d6d1056b7413a192a57457eab7d0d27aba7c23/bloch_quantitative_certificates/bloch_shear_prym_area_core.py\#L204-L278}{For the remaining annular sectors}, one may use the parametrizations
\[
\xi=e^{2\pi i j/3}
\left(1-\frac14as^6\right)
\exp\left(\pm\frac{\pi i}{3}bs^6\right),
\qquad j=0,1,2,
\]
where $(a,b)=(1,t)$ or $(t,1)$ and $(s,t)\in[0,1]^2$.
Figure~\ref{fig:quadrature-partition} shows the twelve outer patches;
the dashed circle $|\xi|=r_c$ is the contour used for the Cauchy bounds.
\begin{figure}[H]
\centering
\begingroup
\let\small\prooffigurefont
\begin{tikzpicture}[x=1cm,y=1cm,font=\small,
  line cap=round,line join=round,>=Latex]
  \path[use as bounding box] (-4.25,-3.18) rectangle (5.48,2.78);
  \fill[proofBlue!12] (0,0) circle (2.4);
  \foreach \root in {0,120,240}{
    \foreach \sgn in {-1,1}{
      \pgfmathsetmacro{\tipangle}{\root+60*\sgn}
      \path[fill=proofTeal!16]
        (\root:2.4)
        -- plot[domain=0:60,samples=41,variable=\ang]
        ({2.4*(1-\ang/240)*cos(\root+\sgn*\ang)},
         {2.4*(1-\ang/240)*sin(\root+\sgn*\ang)})
        arc[start angle=\tipangle,end angle=\root,radius=1.8]
        -- cycle;
    }
  }
  \fill[white] (0,0) circle (1.8);
  \fill[proofBlue!3] (0,0) circle (1.8);
  \foreach \ang in {0,60,...,300}{
    \draw[proofInk!35,line width=.45pt] (\ang:1.8)--(\ang:2.4);
  }
  \foreach \root in {0,120,240}{
    \foreach \sgn in {-1,1}{
      \draw[proofInk!50,line width=.55pt]
        plot[domain=0:60,samples=41,variable=\ang]
        ({2.4*(1-\ang/240)*cos(\root+\sgn*\ang)},
         {2.4*(1-\ang/240)*sin(\root+\sgn*\ang)});
    }
  }
  \draw[proofInk!75,line width=.8pt] (0,0) circle (2.4);
  \draw[proofInk!65,line width=.7pt] (0,0) circle (1.8);
  \draw[proofGold!85,dash pattern=on 2.2pt off 2.2pt,line width=.6pt]
    (0,0) circle (2.1);

  \node[proofInk] at (0,.19) {power series};
  \node[proofInk!80] at (0,-.20) {$|\xi|\leq r_i$};
  \draw[proofInk!45,line width=.45pt]
    (175:2.27)--(-2.73,.82);
  \node[proofInk,anchor=east,inner sep=3pt] at (-2.73,.82)
    {quadrature};
  \draw[proofInk!45,line width=.45pt]
    (202:1.8)--(-2.73,-.86);
  \node[proofInk,anchor=east,inner sep=3pt] at (-2.73,-.86)
    {$r_i=3/4$};

  \draw[proofBlue!65,line width=.45pt]
    (29:2.31)--(2.88,1.55)--(3.13,1.55);
  \fill[proofBlue] (29:2.31) circle (.85pt);
  \node[proofBlue,anchor=west,inner sep=3pt] at (3.13,1.55)
    {$(a,b)=(t,1)$};
  \draw[proofTeal!75,line width=.45pt]
    (16:1.94)--(2.88,.56)--(3.13,.56);
  \fill[proofTeal] (16:1.94) circle (.85pt);
  \node[proofTeal,anchor=west,inner sep=3pt] at (3.13,.56)
    {$(a,b)=(1,t)$};

  \draw[proofGold!75,line width=.45pt]
    (322:2.1)--(2.85,-1.48)--(3.13,-1.48);
  \node[proofGold,anchor=west,align=left,inner sep=3pt]
    at (3.13,-1.48) {$r_c=7/8$\\[-1pt]Cauchy bound};

  \foreach \ang in {0,120,240}{
    \fill[proofRed] (\ang:2.4) circle (1.65pt);
  }
  \node[proofRed,anchor=west,inner sep=4pt] at (2.4,0) {$1$};
  \node[proofRed,above left=3pt] at (120:2.4) {$e^{2\pi i/3}$};
  \node[proofRed,below left=3pt] at (240:2.4) {$e^{4\pi i/3}$};
  \node[proofInk!70,font=\footnotesize] at (.55,-2.99)
    {$t\in[0,1]$: quadrature coordinate};
\end{tikzpicture}

\endgroup
\caption{The integration partition.}
\label{fig:quadrature-partition}
\end{figure}
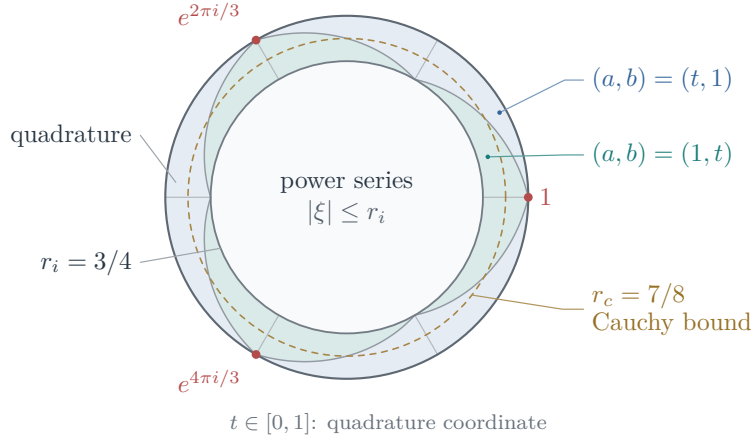

The absolute area Jacobian is
$\frac\pi2(1-\frac14as^6)s^{11}$.
The connection formulas at $\xi^3=1$ put the fractional corner powers
in the form $1-\xi^3=s^6D$ and the flat displacement in the form
$s^2L$. The resulting integrands must be bounded, together with their
required derivatives, on \href{https://github.com/zhangteng2000/Bloch-constant/blob/95d6d1056b7413a192a57457eab7d0d27aba7c23/bloch_quantitative_certificates/bloch_shear_prym_area_audit.py\#L33-L90}{every accepted rectangle}.
For a rectangle of side lengths $\ell_1,\ell_2$, the error bound for
\href{https://github.com/zhangteng2000/Bloch-constant/blob/95d6d1056b7413a192a57457eab7d0d27aba7c23/bloch_quantitative_certificates/bloch_shear_prym_area_result.json\#L2-L23}{the tensor-product four-point Gauss rule} is
\begin{equation}\label{eq:gausserror}
\frac{\ell_1\ell_2}{44100}
\left(
\ell_1^8\sup\left|\frac{\partial_1^8I}{8!}\right|
+\ell_2^8\sup\left|\frac{\partial_2^8I}{8!}\right|
\right).
\end{equation}
The derivative bounds in \eqref{eq:gausserror} may be obtained by
interval automatic differentiation. Equation~\eqref{eq:innertail}
controls the inner-disk tail. 

The periods may be evaluated by a two-point Gauss rule, whose error on
a segment of length $\ell$ is at most
$\ell^5\sup|I^{(4)}/4!|/180$.
For \href{https://github.com/zhangteng2000/Bloch-constant/blob/95d6d1056b7413a192a57457eab7d0d27aba7c23/bloch_quantitative_certificates/bloch_shear_prym_period_audit.py\#L45-L135}{analytic continuation of the square root}, it suffices to enclose
the ratio of $f_E$ to its value at the beginning of each segment in the
right half-plane. \href{https://github.com/zhangteng2000/Bloch-constant/blob/95d6d1056b7413a192a57457eab7d0d27aba7c23/bloch_quantitative_certificates/bloch_shear_prym_period_result.json\#L13-L46}{Paths at heights $h/2$ and $7h/10$} provide an
independent consistency check.

\begin{lem}\label{lem:intervalresult}
The expression for $q_C$ in \eqref{eq:finiteprojection}, with
\eqref{eq:ellipticR} and \eqref{eq:metricweight}, satisfies
\begin{equation}\label{eq:certifiedq}
-\frac{370}{10^6}<q_C<-\frac{337}{10^6}.
\end{equation}
\end{lem}

\href{https://github.com/zhangteng2000/Bloch-constant/blob/95d6d1056b7413a192a57457eab7d0d27aba7c23/bloch_quantitative_certificate_result.json\#L16-L32}{The numerical data reported for this computation} are
\[
\begin{aligned}
D_m&=0.13566064254886\ldots\ \pm3.76\cdot10^{-7},\\
D_mM_*&=(0.09734640938705\ldots\ \pm1.52\cdot10^{-5})\\
&\qquad+i(0.13409159372508\ldots\ \pm1.52\cdot10^{-5}),
\end{aligned}
\]
with both evaluations of $q_C$ in
$[-0.00036984,-0.00033743]$. \href{https://github.com/zhangteng2000/Bloch-constant/blob/95d6d1056b7413a192a57457eab7d0d27aba7c23/bloch_quantitative_certificate.py\#L42-L73}{These values include the factor}
$h/(2\pi)$ in \eqref{eq:finiteprojection}. The resulting interval is contained in
$(-370/10^6,-337/10^6)$, which proves
\eqref{eq:certifiedq}.

\subsection{The explicit bound}

Now we establish the following result.

\begin{prop}\label{prop:quantitative}
Let $t_*=337/20000000$. Then the compact-cylinder family in 
\eqref{eq:compactnormal} satisfies
\begin{equation}\label{eq:quantmain}
\mathcal B\leq b(\widehat f_{t_*})
<\AG\exp\!\left(-\frac{113569}{40000000000000}\right)
<0.471861652113<\AG-\frac{13}{10^{10}}.
\end{equation}
\end{prop}

\begin{proof}
Set $a=337/10^6$ and $t_*=a/20$, so that
$0<t_*<10^{-4}$. The following identity is exact:
\begin{equation*}
-at_*+10t_*^2=-\frac{a^2}{40}
=-d_*,\qquad d_*:=\frac{113569}{40000000000000}.
\end{equation*}
By \eqref{eq:certifiedq} and \eqref{eq:finitebound},
\[
\log\frac{b(\widehat f_{t_*})}{\AG}
<t_*\left(-\frac{337}{10^6}\right)+10t_*^2=-d_*.
\]
This proves the first strict inequality in \eqref{eq:quantmain}.

The remaining scalar comparisons follow, for example, from the
\href{https://github.com/zhangteng2000/Bloch-constant/blob/95d6d1056b7413a192a57457eab7d0d27aba7c23/bloch_quantitative_certificate.py\#L75-L92}{following high-precision bounds for the Gamma function and exponential}:
\begin{equation*}
\begin{aligned}
0.47186165345268178487
&<\AG<0.47186165345268178488,\\
0.47186165211296038375
&<\AG e^{-d_*}<0.47186165211296038376.
\end{aligned}
\end{equation*}
\href{https://github.com/zhangteng2000/Bloch-constant/blob/95d6d1056b7413a192a57457eab7d0d27aba7c23/bloch_quantitative_certificate_result.json\#L33-L42}{These scalar bounds} are independent of the two-dimensional integrals
in \eqref{eq:finiteprojection}. They imply
\[
\AG e^{-d_*}<\frac{471861652113}{10^{12}}
<\AG-\frac{13}{10^{10}}.
\]
The definition of $\mathcal B$ now gives \eqref{eq:quantmain}.
\end{proof}

Proposition~\ref{prop:quantitative} uses the compact-surface
deformation: the coefficient $q_C$ and the remainder estimate in
Proposition~\ref{prop:explicitremainder} refer to that family.
Theorem~\ref{thm:main} uses the single lifted cylinder from
Section~\ref{sec:strip}.

\begin{proof}[Proof of Theorem~\ref{thm:main2}]
    This is an immediate corollary of Proposition~\ref{prop:quantitative}.
\end{proof}

\end{document}